\documentclass[a4paper,fleqn]{cas-sc}

\ExplSyntaxOn
\cs_gset:Npn \__first_footerline:
{
  \group_begin:
  \small
  \sffamily
  \ifnum\theblind>0\relax
  \else
	\__short_authors: ~
  \fi
  { \rmfamily \itshape}
  \group_end:
}
\ExplSyntaxOff

\usepackage[numbers]{natbib}

\usepackage{amsthm}
\usepackage{placeins}
\usepackage{cleveref}
\usepackage{enumitem}
\usepackage{subcaption}
\usepackage{algorithm}
\usepackage{algpseudocode}

\theoremstyle{definition}
\newtheorem{definition}{Definition}
\newtheorem{theorem}{Theorem}

\DeclareMathOperator \supp{supp}
\DeclareMathOperator \GFT{GFT}
\DeclareMathOperator \IGFT{IGFT}
\DeclareMathOperator \hann{hann}
\DeclareMathOperator \diag{diag}
\DeclareMathOperator \RFFT{RFFT}

\begin{document}
\let\WriteBookmarks\relax
\def\floatpagepagefraction{1}
\def\textpagefraction{.001}

\shorttitle{Graph Iterative Filtering methods}

\shortauthors{G. Scarlato, A. Cicone, and M. Donatelli}

\title [mode = title]{Graph Iterative Filtering methods for the analysis of nonstationary signals on graphs}

\author[1]{Giuseppe Scarlato}[orcid=0009-0009-1166-1860]

\ead{gscarlato@uninsubria.it}

\ead[url]{https://giuseppe499.github.io/}

\affiliation[1]{organization={Dipartimento di Scienza e Alta Tecnologia, Università dell'Insubria},
            addressline={Via Valleggio n.11},
            city={Como},
            postcode={22100},
            country={Italy}}

\author[2,3,4]{Antonio Cicone}[orcid=0000-0002-8107-9624]

\cormark[2]

\ead{antonio.cicone@univaq.it}

\ead[url]{https://people.disim.univaq.it/~antonio.cicone/}

\affiliation[2]{organization={Department of Engineering and Computer Science and Mathematics, Università degli Studi dell'Aquila},
            addressline={via Vetoio n.1},
            city={L'Aquila},
            postcode={67100},
            country={Italy}}
\affiliation[3]{organization={Istituto di Astrofisica e Planetologia Spaziali, INAF},
            addressline={Via del Fosso del Cavaliere 100},
            city={Rome},
            postcode={00133},
            country={Italy}}
\affiliation[4]{organization={Istituto Nazionale di Geofisica e Vulcanologia},
            addressline={Via di Vigna Murata 605},
            city={Rome},
            postcode={00143},
            country={Italy}}

\author[1]{Marco Donatelli}[orcid=0000-0001-7958-9126]

\ead{marco.donatelli@uninsubria.it}

\cortext[2]{Corresponding author}

\begin{abstract}
In the analysis of real-world data, extracting meaningful features from signals is a crucial task. This is particularly challenging when signals contain non-stationary frequency components. The Iterative Filtering (IF) method has proven to be an effective tool for decomposing such signals. However, such a technique cannot handle directly data that have been sampled non-uniformly. On the other hand, graph signal processing has gained increasing attention due to its versatility and wide range of applications, and it can handle data sampled both uniformly and non-uniformly.

In this work, we propose two algorithms that extend the IF method to signals defined on graphs. In addition, we provide a unified convergence analysis for the different IF variants. Finally, numerical experiments on a variety of graphs, including real-world data, confirm the effectiveness of the proposed methods.
In particular, we test our algorithms on seismic data and the total electron content of the ionosphere. Those data are by their nature non-uniformly sampled, and, therefore, they cannot be directly analyzed by the standard IF method.

\end{abstract}


\maketitle

\section{Introduction}
\label{sec:introduction}
Signal processing often requires the decomposition of signals into frequency components.
  A common approach is to use the Fourier transform \cite{bracewell1978fourier}, which is, however, inherently unable to capture time-varying frequencies and amplitudes \cite{Cicone2019}.
Since many real-world signals are composed of non-stationary frequency components with time-varying amplitudes, several methods have been developed to overcome this issue, such as the Short-Time Fourier Transform (STFT) \cite{Cohen1994-STFT, Gabor1947-STFT}, the Wavelet Transform \cite{Daubechies1988-Wavelet,Wavelet}, and the Synchrosqueezing Transform \cite{Synchrosqueezing,DAUBECHIES-Synchrosqueezed}.
The main issue with these methods is that they rely on a fixed basis, which may not adapt well to local features.
For this reason, in 1998, the \emph{Empirical Mode Decomposition}~\mbox{(EMD)} method was proposed \cite{EMD} as an adaptive and data-driven approach for decomposing non-stationary signals.
The main idea of EMD is to extract \emph{Intrinsic Mode Functions}~(IMFs) by iteratively removing the local mean of the signal.
For this method, the local mean is computed by interpolating the local maxima and minima of the signal and averaging the two interpolated curves.

Even if EMD has been widely used and has proven to be effective in many applications, it has some theoretical and practical issues, such as the lack of a rigorous mathematical analysis and the sensitivity to noise \cite{Chui_EMD,Huang-EMD_Math_Problems}.
For this reason, another data-driven iterative method, called \emph{Iterative Filtering}~\mbox{(IF)} was proposed \cite{IF}.
Instead of using interpolation to compute the local mean, IF uses a convolution with a compactly supported window function.
This allows a rigorous mathematical analysis, and a clear understanding of the convergence of the IMFs \cite{ALIF}.
Furthermore, by leveraging the convolution theorem, IF can be implemented efficiently using the \emph{Fast Fourier Transform} (FFT) \cite{FIF}.
Even if EMD and IF have been originally proposed for 1D signals, they have also been extended to higher dimensions \cite{MdMvIF,8447300}.
Furthermore, there have been some attempts to extend those methods to non-Euclidean domains, such as the sphere \cite{Sphere_EMD,SIF}, and, for the EMD method, to graphs \cite{Graph_EMD}.

  Graph signal processing has gained increasing attention in recent years due to its versatility and wide range of applications, such as image processing and machine learning \cite{GSP-review}. It provides a powerful framework for analyzing signals defined on irregular domains, such as social networks, sensor networks, and transportation networks.

  For this reason, it is natural to ask if the IF method can be extended to signals defined on graphs. This is particularly interesting because such a method can be used to decompose signals sampled non-uniformly, like geophysical data sampled using instruments distributed over an area \cite{cesaroni2021ionoring}, or economic data collected on a region \cite{ROMANO2019256}. This is possible by triangulating the sampled points, weighting the edges of the triangulation based on the relative distances, and using the corresponding graph structure.

To understand how to extend the IF method to graphs, let us consider the classical 1D IF algorithm \cite{IF}.
In this case, the IF algorithm uses two main tools: convolution and peak counting for deciding the window size.
Therefore, to extend this algorithm, we need to find some
analogous to these tools on graphs. For what concerns the convolution, it exists in the literature an extension of the Fourier transform to graphs, called \emph{Graph Fourier Transform}~(GFT). With this tool, by mimicking the convolution theorem, we can define an operation that is in some ways similar to the standard convolution~\cite{GFT_convolution}.
Using that, we developed an algorithm that we will
call \emph{Graph Fourier Transform Iterative Filtering}~\mbox{(GFT-IF)}.

An issue that arises with this approach is that we have to define the
windows in the spectral domain, which might not be trivial. Furthermore, by doing so we do not have direct control over the support size of the window in the vertex domain, and it becomes difficult to link the window kernel choice to the signal and, more importantly, to the number of extrema.

To address this issue, we also propose another method called \emph{Distance Based Iterative Filtering}~\mbox{(DB-IF)}.
In general, to convolve a signal with a window, we need the ability to translate the window to be centered at each vertex of the graph. This might not be trivial, but we can avoid this problem by directly constructing the windows centered at each vertex based on the distance on the graph. With this approach, we can directly control the
window support size. Therefore, if we have an empirical way to choose the window size, as for the 1D IF case, we can use this algorithm to adaptively decompose a signal on a graph.

  Finally, in this work we give a unified convergence analysis for different IF variants \cite{SIF,MvFIF,MIF,IF}.
  We use this result to give a sufficient condition for the convergence, and to characterize the sequence generated by the inner loop of both the GFT-IF and DB-IF algorithms.

The rest of the work is organized as follows.
In \Cref{sec:IF}, we briefly recall the IF method in both the continuous and discrete settings.
In \Cref{sec:general_IF}, we give a general formulation for the class of IF algorithms on discrete signals, and we give a sufficient condition that guarantees the convergence of the inner loop of those algorithms. This result allows us to unify the convergence analysis of different variants of the IF algorithm and will be used to analyze the convergence of the proposed GFT-IF and DB-IF algorithms.
In \Cref{sec:GFT-IF}, we define the GFT, we introduce the GFT-IF algorithm, and we analyze its convergence.
In \Cref{sec:DB-IF}, we present the DB-IF algorithm, and we analyze its convergence.
  A discussion on the computational cost of the proposed algorithms is given in \Cref{sec:comput_cost}.
Finally, in \Cref{sec:numerical_results}, we present some numerical results comparing the proposed methods and the classical IF algorithm on different graphs and some real-world data.

\section{Iterative Filtering}
\label{sec:IF}
In this section, we briefly review the main features and theoretical results related to the IF method, both in a continuous and discrete setting.

\subsection{IF in the continuous setting}
The key idea of the IF method is to decompose a signal $s:\mathbb{R}\to\mathbb{R}$ into a sum of oscillatory components, called Intrinsic Mode Functions~(IMFs), by iteratively removing the moving average of the signal.

The moving average is computed by convolving the signal $s$ with a window function ${w:\mathbb{R}\to\mathbb{R}}$.
The window function $w$ is chosen to have the following properties:
\begin{itemize}
  \item $w(x) \geq 0$ for all $x\in\mathbb{R}$,
  \item $w(x) = w(-x)$ for all $x\in\mathbb{R}$,
  \item $w\in C^0(\mathbb{R})$,
  \item $\supp(w) = [-l,l]$ for some $l>0$,
  \item $\int_{\mathbb{R}} w(x) dx = \int_{-l}^{l} w(x) dx = 1$.
\end{itemize}
Furthermore, to guarantee the convergence of the IF algorithm, $w$ is chosen such that
\begin{equation*}
  w = v * v,
\end{equation*}
where $*$ is the standard convolution operator and $v:\mathbb{R}\to\mathbb{R}$ has the same properties as $w$, but with $\supp(v) = \left[-\frac{l}{2},\frac{l}{2}\right]$.

The pseudo-code of the IF algorithm is shown in \Cref{algo:IF}.
It consists of an outer loop that extracts the IMFs one by one, and an inner loop that computes each IMF by iteratively removing the moving average of the signal. This inner loop is usually called \emph{sifting process}.

\begin{algorithm}[h!]
  \caption{\textbf{Iterative Filtering} IMF = IF$(s)$}\label{algo:IF}
  \begin{algorithmic}[1]
    \State IMF = $\left\{\right\}$
    \While{the number of extrema of $s$ $\geq 2$}
    \State $s_0 = s$
    \State compute the filter length $l$ for $s$
    \State generate the window $w_l:\mathbb{R}\to\mathbb{R}$ with $\supp(w_l)=[-l,l]$
    \State $m=0$
    \While{the stopping criterion is not satisfied}
    \State  $s_{m+1} = s_{m} - s_{m} * w_l$
    \State  $m = m+1$
    \EndWhile
    \State IMF = IMF$\,\cup\,\{ s_{m}\}$
    \State $s=s-s_{m}$
    \EndWhile
    \State IMF = IMF$\,\cup\,  \{ s\}$
  \end{algorithmic}
\end{algorithm}

\subsection{IF in the discrete setting}
The IF algorithm can be easily extended for discrete signals $s:\{0,\ldots,n-1\}\to\mathbb{R}$ by replacing the continuous convolution with the discrete convolution.
In this case, the window $w:\{0,\ldots,n-1\}\to\mathbb{R}$ is chosen with similar properties as in the continuous case:
\begin{itemize}
  \item $w_i \geq 0$ for all $i=0,\ldots,n-1$,
  \item $w_i = w_{n-i}$ for all $i=1,\ldots,n-1$,
  \item $\sum_{i=0}^{n-1} w_i = \|w\|_1 = 1$,
\end{itemize}
As for the continuous case, a simple way to guarantee the convergence of the inner loop of the IF algorithm is to choose the window $w$ as the self-convolution of a base window $v$:
\begin{equation*}
  w = v * v.
\end{equation*}

\section{General discrete IF algorithm and its convergence}
\label{sec:general_IF}
We observe that several variants of the IF algorithm \cite{SIF,MvFIF,MIF,IF} usually only differ in the way the sifting process is performed (i.e., the inner loop update).
We can leverage this observation to give a general formulation of the IF algorithm for discrete signals, and to unify the convergence analysis of those variants.
In general, we can write a meta algorithm for the IF method as shown in \Cref{algo:general_IF}.

\begin{algorithm}[h!]
  \caption{\textbf{General discrete IF algorithm} IMF = IF$(s)$}\label{algo:general_IF}
  \begin{algorithmic}[1]
    \Require $X$ is a finite set of indices with $|X|=n$.
    \Require $s:X\to\mathbb{R}$ is a signal (i.e., $s\in\mathbb{R}^{n}$).
    \State IMF = $\left\{\right\}$
    \While{the number of extrema of $s$ $\geq 2$}
    \State $s_0 = s$
    \State choose $W\in\mathbb{R}^{n\times n}$ based on $s$.
    \State $m=0$
    \While{the stopping criterion is not satisfied}
    \State  $s_{m+1} = s_{m} - W s_{m}$
    \State  $m = m+1$
    \EndWhile
    \State IMF = IMF$\,\cup\,  \{ s_{m}\}$
    \State $s=s-s_{m}$
    \EndWhile
    \State IMF = IMF$\,\cup\,  \{ s\}$
  \end{algorithmic}
\end{algorithm}

By choosing different sets $X$ and different linear operators $W$, we obtain different variants of the IF algorithm. For example, if we choose $X = \{0,\ldots,n-1\}$ and $W$ as the convolution matrix associated to a window $w\in\mathbb{R}^n$, we obtain the standard IF algorithm \cite{IF}.
If we choose $X = \{0,\ldots,n-1\} \times \{0,\ldots,m-1\}$ and $W$ as the convolution matrix associated to a 2D window $w\in\mathbb{R}^{n\times m}$, we obtain the FIF2 algorithm.
In a similar way, we can also obtain the \emph{Multivariate Fast Iterative Filtering}~(MvFIF) \cite{MvFIF}, the \emph{Multidimensional Iterative Filtering}~(MIF) \cite{MIF}, and the \emph{Spherical Iterative Filtering}~(SIF) \cite{SIF} algorithms.

As proven by the following theorem, we can give a sufficient condition on the linear operator $W$ that guarantees the convergence of the inner loop of \Cref{algo:general_IF}, and we can also characterize the limit of the generated sequence $\left\{s_m\right\}$.

\begin{theorem}[General discrete IF convergence]
  \label{thm:general_IF_convergence}
  Let $W\in\mathbb{R}^{n\times n}$ such that
  \begin{enumerate}[label={\arabic*)}]
    \item $WW^H = W^HW$ ($W$ is a normal matrix),\label{item:normal_general_IF_convergence}
    \item $|1-\lambda_i| < 1 \vee \lambda_i = 0$ for all eigenvalues $\lambda_i$ of $W$.\label{item:eigen_general_IF_convergence}
  \end{enumerate}
  Then the inner loop of the general discrete IF algorithm (\Cref{algo:general_IF}) converges to
  \begin{equation*}
    s_\infty = Q D Q^H s_0,
  \end{equation*}
  where $Q$ is a unitary matrix such that $W = Q T Q^H$ with $T$ a diagonal matrix, and $D$ is a diagonal matrix defined as
  \begin{equation*}
    D_{i,i} =
    \begin{cases}
      1 & \text{if } \lambda_i = T_{i,i} = 0, \\
      0 & \text{if } |1-\lambda_i| < 1.
    \end{cases}
  \end{equation*}
\end{theorem}
\begin{proof}
  Using Schur decomposition, we can write $W = Q T Q^H$ where $Q$ is a unitary matrix and $T$ is an upper triangular matrix. Furthermore, since $W$ is normal, we have that $T$ is a diagonal matrix.
  From \Cref{algo:general_IF}, we have that
  \begin{equation*}
    s_{m} = s_{m-1} - W s_{m-1} = (I - W) s_{m-1}.
  \end{equation*}
  Iterating this relation and substituting the Schur decomposition of $W$, we obtain
  \begin{equation*}
    s_{m} = (I - W)^m s_0 = Q (I - T)^m Q^H s_0.
  \end{equation*}
  We observe that $I - T$ is a diagonal matrix with diagonal entries $1 - \lambda_i$, where $\lambda_i=T_{i,i}$ are the eigenvalues of $W$. Therefore, the sequence $\left\{s_m\right\}$ converges if and only if
  \begin{equation*}
    |1-\lambda_i| < 1 \vee \lambda_i = 0 \quad \forall\ i.
  \end{equation*}
  In particular, if this condition is satisfied, we have that
  \begin{equation*}
    s_\infty = \lim_{m\to\infty} s_{m} = \lim_{m\to\infty} Q (I - T)^m Q^H s_0 =
    Q D Q^H s_0,
  \end{equation*}
  with
  \begin{equation*}
    D_{i,i} =
    \begin{cases}
      1 & \text{if } \lambda_i = T_{i,i} = 0, \\
      0 & \text{if } |1-\lambda_i| < 1.
    \end{cases}
  \end{equation*}
\end{proof}

In practical applications, an easy way to guarantee the convergence of the inner loop of \Cref{algo:general_IF} is to construct the
window matrix $W$ as the self-product of a base window matrix $B$, which is
symmetric and stochastic. This ensures that the window matrix $W$ is also
stochastic and symmetric. Furthermore, $W$ is a positive semi-definite matrix
\begin{equation}
  x^T W x = x^T B B x = (B x)^T B x \geq 0 \quad \forall\ x\in\mathbb{R}^n.
\end{equation}
Since $W$ is stochastic, Greshgorin's circle theorem implies that all the eigenvalues of $W$ are contained in the unit disk and, since $W$ is positive semi-definite, all the eigenvalues are real and non-negative. This means that $\lambda_i \in [0,1]$ for all eigenvalues $\lambda_i$ of $W$.
Therefore, the conditions of \Cref{thm:general_IF_convergence} are satisfied.
We also observe that requiring that the window matrix $W$ is stochastic ensures that the energy of the moving average is not increased compared to the energy of the signal in the sense that
\begin{equation*}
  \|W\|_1 = 1 \wedge \|W\|_\infty = 1.
\end{equation*}

\section{Graph Fourier Transform Iterative Filtering (GFT-IF)}
\label{sec:GFT-IF}
Our new GFT-IF algorithm is based on the \emph{Graph Fourier Transform}, hence we first introduce the necessary mathematical tools.

\subsection{Graph Fourier Transform and Graph Convolution}
\begin{definition}[Undirected weighted graph]
  An undirected weighted graph $G$ is a tuple ${G = (V,E,w)}$ with $V$ the set of
  vertices, $E$ the set of edges containing unordered pairs of vertices, and $w:E\to\mathbb{R}$ the weight function.
\end{definition}
\begin{definition}[Adjacency Matrix]
  Let $G = (V,E,w)$ be an undirected weighted graph. The adjacency
  matrix $A$ of $G$ is defined as
  \begin{equation}
    A_{i,j} =
    \begin{cases}
      w(i,j) & \text{if } (i,j)\in E, \\
      0      & \text{otherwise}.
    \end{cases}
  \end{equation}
\end{definition}
\begin{definition}[Degree Matrix]
  Let $G = (V,E,w)$ be an undirected weighted graph. The degree
  matrix $D$ of $G$ is a diagonal matrix defined as
  \begin{equation}
    D_{i,i} = \sum_{j\in V} A_{i,j}.
  \end{equation}
\end{definition}

\begin{definition}[Graph Laplacian]
  Let $G = (V,E,w)$ be an undirected weighted graph. The graph
  Laplacian $L$ of $G$ is defined as
  \begin{equation}
    L = D - A.
  \end{equation}
  Note that $L$ is a real symmetric positive semi-definite matrix. Therefore, it is diagonalizable by an orthogonal matrix and all its eigenvalues are real and non-negative.
\end{definition}

\begin{definition}[Graph Fourier Transform]
  Let $G = (V,E,w)$ be an undirected weighted graph. The graph
  Fourier Transform of a signal $s:V\to\mathbb{R}$ is defined as
  \begin{equation}
    \hat{s} = \GFT(s) = U^T s,
  \end{equation}
  where $U$ is an orthonormal matrix such that
  \begin{equation}
    L = U\Lambda U^T
  \end{equation}
  and $\Lambda$ is a diagonal matrix.

  Similarly, the inverse Graph Fourier Transform of a signal $\hat
  s:V\to\mathbb{R}$ is defined as
  \begin{equation}
    s = \IGFT(\hat{s}) = U \hat{s}.
  \end{equation}
\end{definition}

\begin{definition}[Graph Convolution]
  \label{def:graph_convolution}
  Let $G = (V,E,w)$ be an undirected weighted graph. The graph convolution of the
  signals $s,v:V\to\mathbb{R}$ is defined as
  \begin{equation}
    s \circledast v = \IGFT\left(\hat s \odot \hat v\right) = U \left[(U^T s) \odot (U^T v)\right]
  \end{equation}
  where $\odot$ is the Hadamard product.
\end{definition}

\pagebreak[3]

\subsection{GFT-IF Algorithm}
The basic idea of the GFT-IF algorithm is to use the graph convolution (\Cref{def:graph_convolution}) to compute the moving average of the signal in the inner loop of the IF algorithm.
The main issue that arises with this approach is that we have to define the convolution kernel $\hat w$ in the spectral~domain~(GFT~basis), which might not be trivial. In particular, by doing so, we do not have direct control over the support size of the window in the vertex domain, and it becomes difficult to link the window kernel choice to the number of extrema of the signal.
In the numerical results (\Cref{sec:numerical_results}), we will use a simple low-pass filter as a convolution kernel.
\Cref{algo:GFT-IF} shows the pseudo-code of the proposed GFT-IF algorithm.

We observe that, by defining $\hat s_m = U^T s_m$, the inner loop of the GFT-IF algorithm can be rewritten as
\begin{equation*}
  \hat s_{m+1} = \hat s_{m} - \hat w \odot  \hat s_{m}.
\end{equation*}
Therefore, the inner loop can be computed directly in the spectral domain, at a reduced computational cost of $O(|V|)$ flops per iteration, where $|V|$ denotes the number of vertices of the graph.
We will go over the computational cost of \Cref{algo:GFT-IF} in more details in \Cref{sec:comput_cost}.
\begin{algorithm}[H]
  \caption{\textbf{Graph Fourier Transform Iterative Filtering} GFT-IMF = GFT-IF$(s)$}\label{algo:GFT-IF}
  \begin{algorithmic}[1]
    \Require $G = (V,E,w)$ is an undirected weighted graph.
    \Require $L$ is the graph Laplacian of $G$.
    \Require $s:V\to\mathbb{R}$ is a signal on $G$.
    \State Compute the spectral decomposition of $L = UT U^T$.
    \State GFT-IMF = $\left\{\right\}$
    \While{the number of extrema of $s$ $\geq 2$}
    \State $s_0 = s$
    \State  select a convolution kernel $\hat w:V\to\mathbb{R}$ based on $s$
    \State $m=0$
    \While{the stopping criterion is not satisfied}
    \State  $s_{m+1} = s_{m} - s_{m} \circledast w = s_{m} - U(\hat w \odot U^T s_{m})$
    \State  $m = m+1$
    \EndWhile
    \State GFT-IMF = GFT-IMF$\,\cup\,\{ s_{m}\}$
    \State $s=s-s_{m}$
    \EndWhile
    \State GFT-IMF = GFT-IMF$\,\cup\,  \{ s\}$
  \end{algorithmic}
\end{algorithm}

\subsection{Inner loop convergence of the GFT-IF algorithm}
Based on \Cref{thm:general_IF_convergence}, we can give a sufficient condition on the convolution kernel $\hat w$ that guarantees the convergence of the inner loop of \Cref{algo:GFT-IF}, and we can also characterize the limit of the generated sequence $\left\{s_m\right\}$.

\begin{theorem}[GFT-IF convergence]
  \label{thm:GFT-IF_convergence}
  Let $\hat w:V\to\mathbb{R}$ be a convolution kernel such that
  \begin{equation*}
    \hat w_i \in [0,2) \quad \forall\ i,
  \end{equation*}
  then the inner loop of the GFT-IF algorithm (\Cref{algo:GFT-IF}) converges to
  \begin{equation*}
    s_\infty = w^{(\infty)} \circledast s_0,
  \end{equation*}
  where $w^{(\infty)} = \IGFT(\hat w^{(\infty)})$ and $\hat w^{(\infty)}$ is defined as
  \begin{equation*}
    \hat w^{(\infty)}_i =
    \begin{cases}
      1 & \text{if } \hat w_i = 0, \\
      0 & \text{if } 0 < \hat w_i < 2.
    \end{cases}
  \end{equation*}
\end{theorem}
\begin{proof}
  We observe that the GFT-IF algorithm (\Cref{algo:GFT-IF}) is a particular case of the general discrete IF algorithm (\Cref{algo:general_IF}) with $W$ defined as
  \begin{equation*}
    W = U \diag(\hat w)\, U^T.
  \end{equation*}
  where $U^T$ is the GFT matrix of the graph $G$.
  Therefore, the matrix $W$ satisfies hypothesis \ref{item:normal_general_IF_convergence} of \Cref{thm:general_IF_convergence}.
  Given that $\hat w_i \in [0,2)\ \forall i$, hypotheses \ref{item:eigen_general_IF_convergence} is also satisfied and, therefore, the inner loop of \Cref{algo:GFT-IF} converges to
  \begin{equation*}
    s_\infty = U \diag(\hat w^{(\infty)})\, U^T s_0 = w^{(\infty)} \circledast s_0,
  \end{equation*}
  with
  \begin{equation*}
    \hat w^{(\infty)}_i =
    \begin{cases}
      1 & \text{if } \hat w_i = 0, \\
      0 & \text{if } 0 < \hat w_i < 2.
    \end{cases}
  \end{equation*}
\end{proof}

In practical applications, it is useful to choose the convolution kernel $\hat w$ such that
\begin{equation*}
  \|\hat w\|_\infty = 1.
\end{equation*}
This ensures that the energy of the moving average is not increased compared to the energy of the signal, in the sense that it holds
\begin{equation*}
  \|s \circledast w\|_2 = \|\hat s \odot \hat w\|_2 \leq \|\hat s\|_2 \|\hat w\|_\infty = \|\hat s\|_2 = \|s\|_2.
\end{equation*}

\section{Distance Based Iterative Filtering (DB-IF)}
\label{sec:DB-IF}
In this section, we propose an algorithm that does not directly use the structure of the given graph $G$.
In fact, we only need a matrix $C$ such that the entry $C_{i,j}$ represents some kind of distance between the vertices $i$ and $j$ of the graph.
Of course, if we have a weighted graph, we
can construct a matrix $C$ based on the shortest
path distance.

For the numerical results in
\Cref{sec:numerical_results}, for each example, we used the distance defined on the spaces in which the graphs are embedded. As an example, in the case of 2D data (\Cref{sec:example_2}), we used the Euclidean distance.
Therefore, in a way, the graph structure is ignored by
the algorithm, and instead we use a more complex structure obtained by embedding
the graph in $\mathbb{R}^2$.
The pseudo-code of the proposed DB-IF method is shown in \Cref{algo:DB-IF}.

\begin{algorithm}[h!]
  \caption{\textbf{Distance Based Iterative Filtering} DB-IMF = DB-IF$(s)$}\label{algo:DB-IF}
  \begin{algorithmic}[1]
    \Require $G = (V,E)$ is a graph.
    \Require $s:V\to\mathbb{R}$ is a signal on $G$.
    \Require $C$ is a matrix representing the distance between the vertices of $G$.
    \Require $w:\mathbb{R}\to\mathbb{R}$ is a window function with $\supp(w)=[-1,1]$, and $w_l(x)=w(x/l)\quad \forall\ l>0$.
    \State DB-IMF = $\left\{\right\}$
    \While{the number of extrema of $s$ $\geq 2$}
    \State $s_0 = s$
    \State compute the filter window length $l$ for $s$
    \State compute the base window matrix $B_{i,j} = w_{l}(C_{i,j})$
    \State normalize the rows of $B$ to sum to 1
    \State $m=0$
    \While{the stopping criterion is not satisfied}
    \State  $s_{m+1} = s_{m} - W s_{m} = s_{m} - B B s_{m}$
    \State  $m = m+1$
    \EndWhile
    \State DB-IMF = DB-IMF$\,\cup\,  \{ s_{m}\}$
    \State $s=s-s_{m}$
    \EndWhile
    \State DB-IMF = DB-IMF$\,\cup\,  \{ s\}$
  \end{algorithmic}
\end{algorithm}

As we can see from \Cref{algo:DB-IF}, for each step of the algorithm, we need to multiply the signal by the window matrix $W$. This normally has a computational cost of $O(|V|^2)$. If the window size is small compared to the size of the graph, then the window matrix~$W$ will usually be sparse and, therefore, the cost of the multiplication can be reduced to~$O(|V|)$. We will review in more details the computational cost of \Cref{algo:DB-IF} in \Cref{sec:comput_cost}.

\subsection{Inner loop convergence of the DB-IF algorithm}
Based on \Cref{thm:general_IF_convergence}, we can show that the inner loop of the DB-IF algorithm converges, and we can also characterize the limit of the generated sequence $\left\{s_m\right\}$.

\begin{theorem}
  \label{thm:DB-IF_convergence}(DB-IF convergence)
  Let $W$ be defined as in \Cref{algo:DB-IF},
  then the inner loop of the DB-IF algorithm (\Cref{algo:DB-IF}) converges to
  \begin{equation*}
    s_\infty = Q D Q^T s_0,
  \end{equation*}
  where $Q$ is an orthogonal matrix such that $W = Q \Lambda Q^T$ with $\Lambda$ a diagonal matrix, and $D$ is a diagonal matrix defined as
  \begin{equation*}
    D_{i,i} =
    \begin{cases}
      1 & \text{if } \lambda_i = \Lambda_{i,i} = 0, \\
      0 & \text{otherwise}.
    \end{cases}
  \end{equation*}
\end{theorem}
\begin{proof}
  From \Cref{algo:DB-IF}, we know that
  \begin{equation*}
    W = B B,
  \end{equation*}
  where $B$ is a symmetric and stochastic matrix.
  As seen in \Cref{sec:general_IF}, this implies that $W$ is symmetric and all the eigenvalues of $W$ are contained in the interval $[0,1]$. Therefore, all the hypotheses of \Cref{thm:general_IF_convergence} hold and the thesis follows.
  Furthermore, since $W$ is real and symmetric, we can choose $Q$ not only unitary but also orthogonal, i.e., $Q^T = Q^H = Q^{-1}$.
\end{proof}

\section{Computational cost analysis}
\label{sec:comput_cost}
In this section, we analyze the computational cost of the proposed GFT-IF and DB-IF algorithms.
As a baseline, we also compare those methods with the classical FIF algorithm.
We examine the key steps for which the three algorithms differ: precomputation, window computation, sifting iteration, and, for GFT-IF and FIF, transformation from the spectral domain to the vertex domain.

In this section, we use $n$ to denote the number of vertices of the graph, i.e., $n=|V|$.

\subsection{Precomputation}
\label{sec:precomp}
Precomputation includes tasks that can be performed once and reused for all signals defined on the same graph.

For GFT-IF, this consists of computing the full eigendecomposition of the graph Laplacian, which is needed to compute the GFT and IGFT.
This step has a computational cost of $O(n^3)$.
To reduce the computational cost of this step, different strategies could be used, such as computing only a subset of the eigenvectors, or using approximate methods for eigendecomposition (see, \cite{Marini2025, ALEOTTI202443}).

For DB-IF, this consists of computing the distance matrix $C$. The cost of this step depends on how the distance matrix $C$ is obtained.
If it is obtained by embedding the graph in a space and using the corresponding distance, the cost is $O(n^2)$.
If $C$ is instead obtained by computing the shortest-path distances, it can be done using the Floyd-Warshall algorithm with a
computational cost of $O(n^3)$ or, if the graph is sparse, using Dijkstra's
algorithm with a Fibonacci heap with a computational cost of ${O[n(|E| +
n\log n)]=O(n^2 \log n)}$, see \cite{Floyd–Warshall,Dijkstra-Fibonacci}.

Note that the classical FIF method requires no precomputation since the DFT matrix is fixed.

\subsection{Window computation}
In this section, we analyze the cost of computing the window used in the sifting process for each method.

For GFT-IF, this consists of selecting a convolution kernel $\hat w$ in the spectral domain (line 5 of \Cref{algo:GFT-IF}).
For the numerical results in \Cref{sec:numerical_results}, we simply choose a low-pass filter as a convolution kernel computed as
\begin{equation*}
  \left(\hat w_l\right)_i = \hann(\lambda_i/l),
\end{equation*}
where $\hann$ is the Hann window function and $\lambda_i$ is the $i$-th smallest eigenvalue of the graph Laplacian.
This step has a computational cost of $O(n)$.

For DB-IF, this consists of computing the base window matrix $B$ (lines 5-6 of \Cref{algo:DB-IF}) as
\begin{equation*}
  B_{i,j} = \frac{w(C_{i,j}/l)}{\sum_{k=1}^n w(C_{i,k}/l)},
\end{equation*}
where $w$ is the window function and $l$ is the window length. This step has a computational cost of $O(n^2)$.

For the classical FIF method, this consists of computing the base window in the vertex domain
\begin{equation*}
  (v_l)_i = v(i/l),
\end{equation*}
where $v$ is a window function and $l$ is the window length. This step has a computational cost of $O(n)$.
Furthermore, since the sifting process is performed in the spectral domain, it is also necessary to compute the \emph{Real Fast Fourier Transform } (RFFT) of the base window $\hat v = \RFFT(v)$, which has a computational cost of $O(n \log n)$.
Finally, we compute the window kernel
\begin{equation*}
  \hat w = \hat v \odot \hat v,
\end{equation*}
which has a computational cost of $O(n)$.
In total, the window computation for FIF has a computational cost of $O(n \log n)$.

\subsection{Sifting process step}
As seen in \Cref{sec:general_IF}, the main difference between the three methods consists in the operator $W$ used in the sifting process.

For GFT-IF and FIF, the operator $W$ is defined as a convolution operator and, for efficiency reasons, the sifting process is performed in the spectral domain.
This reduces the product $W s_m$ to an element-wise product $\hat w \odot \hat s_m$, which has a computational cost of $O(n)$.

For DB-IF, the operator $W$ is defined as a matrix product with the window matrix $W = B B$. This has a computational cost of $O(n^2)$.
If the support size $l$ of the window is small compared to the size of the graph (i.e., $l \ll \max_{i,j} C_{i,j}$), and the graph is sparse, then the window matrix $W$ will be sparse and the computational cost can be reduced to $O(n)$.

\subsection{Transformation from spectral domain to vertex domain}
Since the sifting process for GFT-IF and FIF is performed in the spectral domain, at the end of the inner loop, it is necessary to transform the obtained IMF from the spectral domain to the vertex domain.
For GFT-IF, this consists of computing the IGFT of the obtained IMF, which has a computational cost of $O(n^2)$.
For FIF, this consists of computing the Inverse Real Fast Fourier Transform (IRFFT) of the obtained IMF, which has a computational cost of $O(n \log n)$.

\subsection{Summary of computational cost}
\label{sec:summary_of_comp_cost}
To summarize, the precomputation cost for the proposed methods is:
\begin{enumerate}[label= \arabic*$^p$)]
  \item GFT-IF: $O(n^3)$, \label{item:pre_GFT}
  \item DB-IF: $O(n^2)$. \label{item:pre_DB}
\end{enumerate}
We note that GFT-IF has the highest cost due to eigendecomposition of the graph Laplacian,  which could be approximated cheaply using different strategies (see, \cite{Marini2025, ALEOTTI202443}).

Assuming that the number of iterations needed for convergence of the inner loop is a constant $m$, and that the extracted IMFs are $k$ for each method, the total cost of the three methods excluding precomputation is:
\begin{enumerate}[label= \arabic*)]
  \item GFT-IF: $O(k m n + k n^2)$, \label{item:cost_GFT}
  \item DB-IF: $O(k m n^2)$, \label{item:cost_DB}
  \item FIF: $O(k n \log n + k m n)$. \label{item:cost_FIF}
\end{enumerate}
We observe that, in this case, DB-IF is the most expensive method due to the cost of the matrix product with the window matrix $W$ in the sifting process.
GFT-IF has a lower cost than DB-IF due to the efficiency of the convolution operator in the spectral domain, but it is more expensive than FIF due to the cost of the IGFT at the end of each inner loop.
Despite higher computational cost, GFT-IF and DB-IF have the significant advantage of working on arbitrary graphs, including irregular domains, while FIF is limited to regular grids.

\section{Numerical Results}
\label{sec:numerical_results}

In this section, we present numerical results for the proposed GFT-IF and DB-IF algorithms. We have tested both algorithms on different graphs, such as a random ring graph (\Cref{sec:example_1}) and a triangulation of random points on a 2D square (\Cref{sec:example_2}). Furthermore, we also tested the proposed methods on real-world data such as magnitude time series of the global seismicity (\Cref{sec:example_3}) and the \emph{Total Electron Content} (TEC) of the Earth's ionosphere measured over Italy (\Cref{sec:example_4}).
In the 1D case (\Cref{sec:example_1}), we also compare the results of our methods with the ones obtained with the classical FIF algorithm on the same signal sampled on an equispaced grid of points.
All the tests were performed on a PC running Fedora 43 with an AMD Ryzen 9 3900X 12-core Processor and 32GB of RAM.
All the methods were implemented in Python 3.13 using the NumPy and SciPy libraries.
The codes used for the numerical results are available in a public git repository \footnote{\url{https://github.com/Giuseppe499/GraphIF}}

For all the examples, the convolution kernel $\hat w$ for the GFT-IF algorithm was defined in the spectral domain as
\begin{equation*}
  \left(\hat w_l\right)_i = \hann(\lambda_i/l),
\end{equation*}
where $\hann$ is the Hann window function, $l>0$ is the kernel support size, and $\lambda_i$ is the $i$-th smallest
eigenvalue of the graph Laplacian $L$.

For the DB-IF algorithm, the distance matrix $C$ was defined by considering the vertices of the graph as points of the natural space in which they are defined and by using the corresponding distance. In particular, in the case of a graph representing a grid of points on the interval $[0, 2\pi)$, we have used the distance in radians between the points on the circle; in the case of a graph representing a triangulation of random points on a 2D square, we have used the Euclidean distance.
Then, as described in \Cref{algo:DB-IF}, the window matrix $W$ was defined as the self-product of a base window matrix $B$ defined as
\begin{equation*}
  B_{i,j} = \frac{w_l(C_{i,j})}{\sum_{k=1}^n w_l(C_{i,k})},
\end{equation*}
where $w_l$ was defined as
\begin{equation*}
  w_l(x) = \hann(x/l).
\end{equation*}
In the 1D case, the window size $l$ was defined as
\begin{equation}
  \label{eq:DB_IF_window_length}
  l = 2\nu\frac{b-a}{k},
\end{equation}
where $[a,b]$ is the interval on which the signal is defined, $k$ is the number of extrema of the signal, and $\nu$ is a tuning parameter set to $\nu = 1.6$ for all the examples.
This is similar to what is done in the classical FIF algorithm, where the window size is defined as
\begin{equation*}
  l = 2 \left\lfloor\nu\frac{n}{k}\right\rfloor,
\end{equation*}
where $n$ is the number of samples of the signal, $k$ is the number of extrema of the signal, and $\nu$ is usually set to $1.6$ \cite{ALIF}.

\subsection{Example 1: 1D signal on a random ring graph}
\label{sec:example_1}
In this example, we have tested the GFT-IF and the DB-IF algorithms on a 1D signal on a grid of random points uniformly distributed on the interval $[0, 2\pi)$.
The graph is constructed as a ring graph where each vertex is connected to its 2 nearest neighbors on the left and on the right. The edges are weighted by $1/d(i,j)$ where $d(i,j)$ is the distance on the circle between the vertices $i$ and $j$.
For comparison, we also show the results obtained with the classical FIF algorithm on the same signal sampled on an equispaced grid of points on the interval $[0, 2\pi)$.

The signal consists of the sum of two sinusoidal functions with nonstationary frequencies
\begin{equation*}
  \begin{split}
    b_0(x) &= \frac{1}{2} \sin\left[\left(30 - \frac{5}{2 \pi} x\right) x + 1\right],\\
    b_1(x) &= \sin\left[\left(2 + \frac{2}{2 \pi} x\right) x - 1.5\right],\\
  \end{split}
\end{equation*}
and is shown in \Cref{fig:ex_1signals}. In the same figure, the two components of the signal $b_0$ and $b_1$ are also shown.
\Cref{fig:ex_1GFT_kernel} shows the GFT of
the signal $s$ compared with the convolution kernel used in the GFT-IF
algorithm.
Similarly, \Cref{fig:ex_1FIF_kernel} shows the \emph{Real Fast Fourier Transform} (RFFT) of
the signal $s$ compared with the convolution kernel used for the classical FIF method.
We observe that, in both cases, the convolution kernel is a low-pass filter that preserves the low frequencies of the signal and attenuates or completely removes the high frequencies.

\Cref{fig:ex_1windows} shows some examples of the window functions used in the GFT-IF, DB-IF, and FIF algorithms.
For the GFT-IF algorithm (\Cref{fig:ex_1GFT_windows}) and for the FIF algorithm (\Cref{fig:ex_1FIF_windows}), the windows centered at different vertices of the graph/grid are obtained by taking the convolution of $w$ with the delta signals centered at the corresponding vertices.
For the DB-IF algorithm (\Cref{fig:ex_1DB_windows}), the windows are the rows of the window matrix $W$ used in the algorithm.
We note that for GFT-IF, since the windows are defined in the spectral domain, there are no constraints on their shape in the vertex domain, and, therefore, they have some negative values and do not have local support.
For the DB-IF algorithm, since the window matrix $W$ is stochastic, the windows sum to 1. Moreover, since the support size is the same for each window, windows centered in areas where points are more dense have a small maximum value, while, windows centered in areas where points are more sparse have a large maximum value. The same effect can also be observed for the windows of GFT-IF.

\Cref{fig:ex_1GFT_results,fig:ex_1DB_results,fig:ex_1FIF_results} show the results of the GFT-IF, DB-IF, and FIF algorithms, respectively.
We observe that all three algorithms are able to correctly decompose the signal into its two components, but the error of the decomposition is slightly smaller for FIF compared to GFT-IF and DB-IF. Furthermore, the error of FIF is concentrated on the boundary of the domain, probably due to boundary effects, while the error of GFT-IF and DB-IF is more uniformly distributed along the signal.
The slightly better performance of FIF compared to GFT-IF and DB-IF can be explained by the fact that FIF is working on a regular grid, while GFT-IF and DB-IF are working on a random grid with non-uniformly distributed points.

To conclude on this example, we analyze the computational efficiency of the three algorithms. Similarly to \Cref{sec:summary_of_comp_cost}, we separated the precomputation cost from the total cost of the methods.
\Cref{tab:ex_1CPU_time_precomputation} shows the CPU time for the precomputation of GFT-IF (see \ref{item:pre_GFT}) and DB-IF (see \ref{item:pre_DB}).
We note that the precomputation of GFT-IF is, as expected, significantly more expensive than that of DB-IF due to the cost of the eigendecomposition of the graph Laplacian.

Setting the number of iterations of the inner loop to $m=10$ and the number of extracted IMFs to $k=10$, the total CPU time of the three methods, excluding precomputation, is shown in \Cref{tab:ex_1CPU_time_total_ex_precomputation}. The number of extracted IMFs is set to $k=10$, exceeding the two IMFs extracted in this example, to give a more general and realistic estimate of the computational cost of the methods.

We observe that, in this case, DB-IF is the most expensive method due to the cost of the matrix product with the window matrix $W$ in the sifting process, while GFT-IF has a lower cost than DB-IF due to the efficiency of the convolution operator in the spectral domain.
However, GFT-IF is more expensive than FIF due to the cost of the IGFT at the end of each inner loop, which is more expensive than the cost of the IRFFT used in FIF.
Despite higher computational cost, GFT-IF and DB-IF have the advantage of working on a non-uniform grid, while FIF can only be used on the equispaced grid.

\begin{figure}[h]
  \centering
  \makebox[\textwidth][c]{
    \begin{subfigure}{0.33\textwidth}
      \centering
      \includegraphics[width=\textwidth]{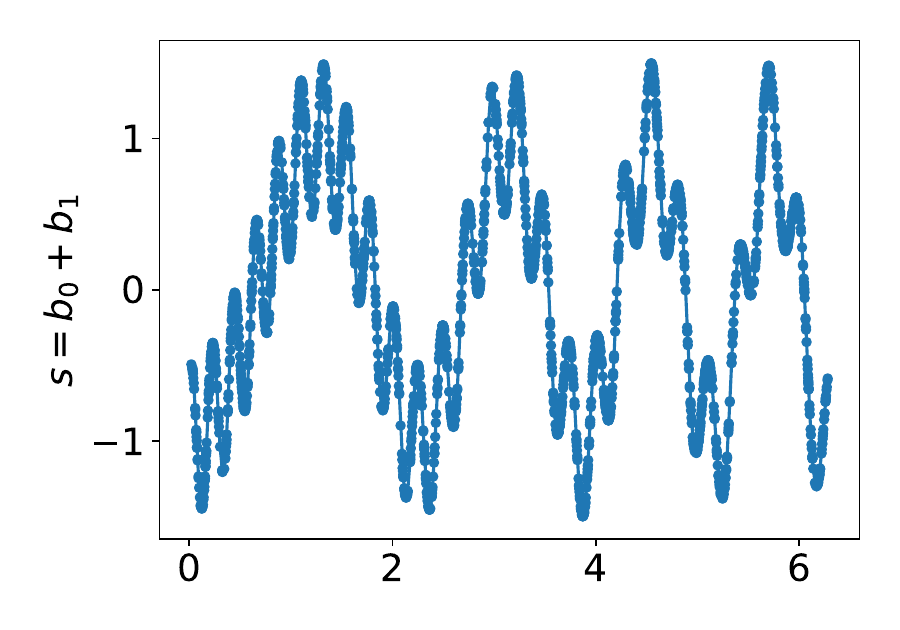}
      \subcaption{\label{fig:ex_1signal}}
    \end{subfigure}
    \begin{subfigure}{0.33\textwidth}
      \centering
      \includegraphics[width=\textwidth]{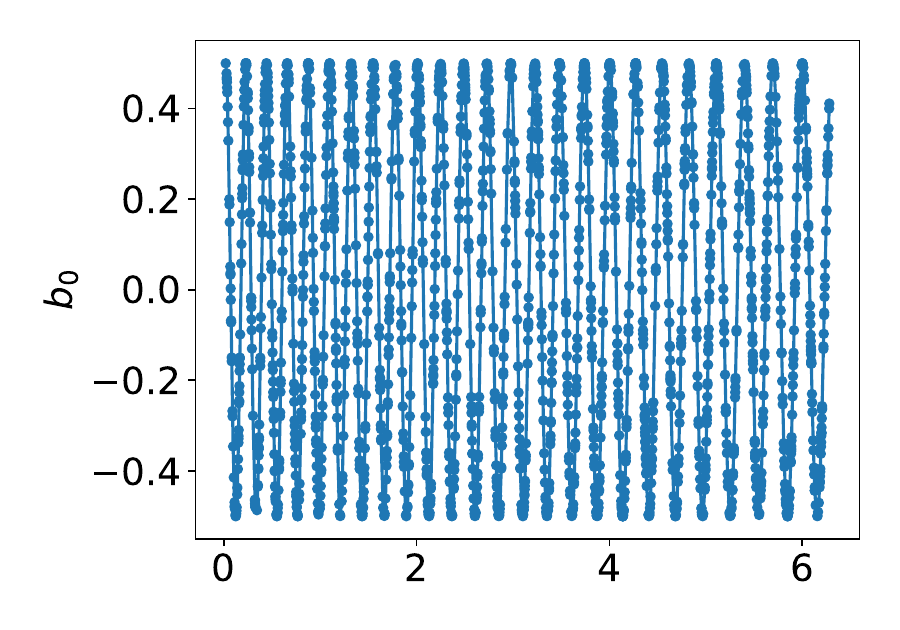}
      \subcaption{\label{fig:ex_1base_signal_0}}
    \end{subfigure}
    \begin{subfigure}{0.33\textwidth}
      \centering
      \includegraphics[width=\textwidth]{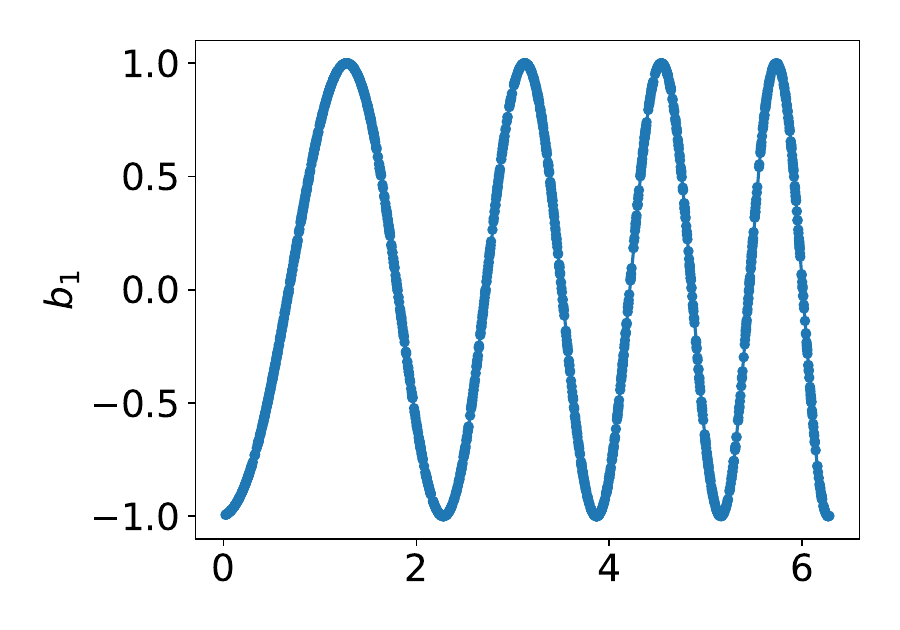}
      \subcaption{\label{fig:ex_1base_signal_1}}
    \end{subfigure}
  }
  \makebox[\textwidth][c]{
    \begin{subfigure}{0.33\textwidth}
      \centering
      \includegraphics[width=\textwidth]{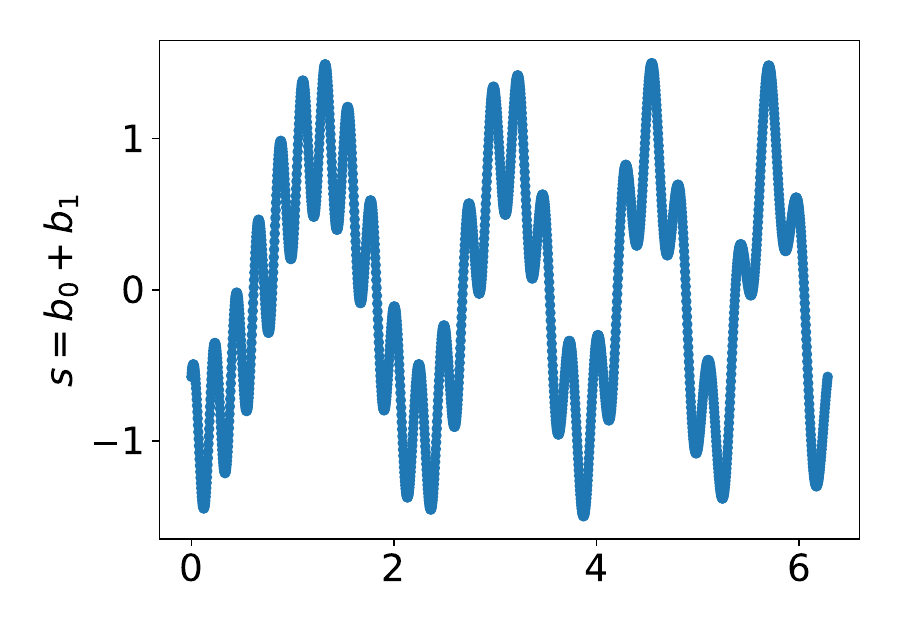}
      \subcaption{\label{fig:ex_1FIF_signal}}
    \end{subfigure}
    \begin{subfigure}{0.33\textwidth}
      \centering
      \includegraphics[width=\textwidth]{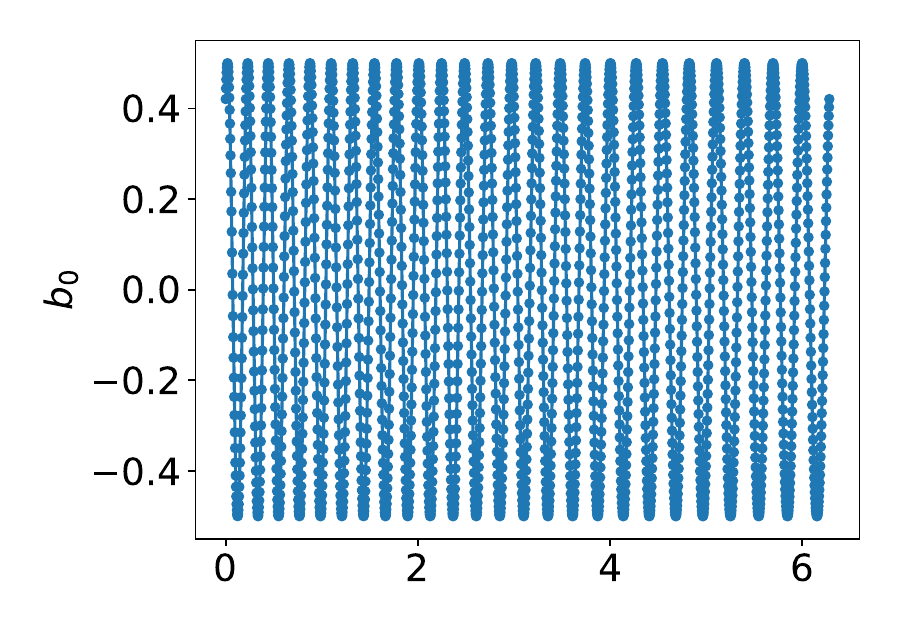}
      \subcaption{\label{fig:ex_1FIF_base_signal_0}}
    \end{subfigure}
    \begin{subfigure}{0.33\textwidth}
      \centering
      \includegraphics[width=\textwidth]{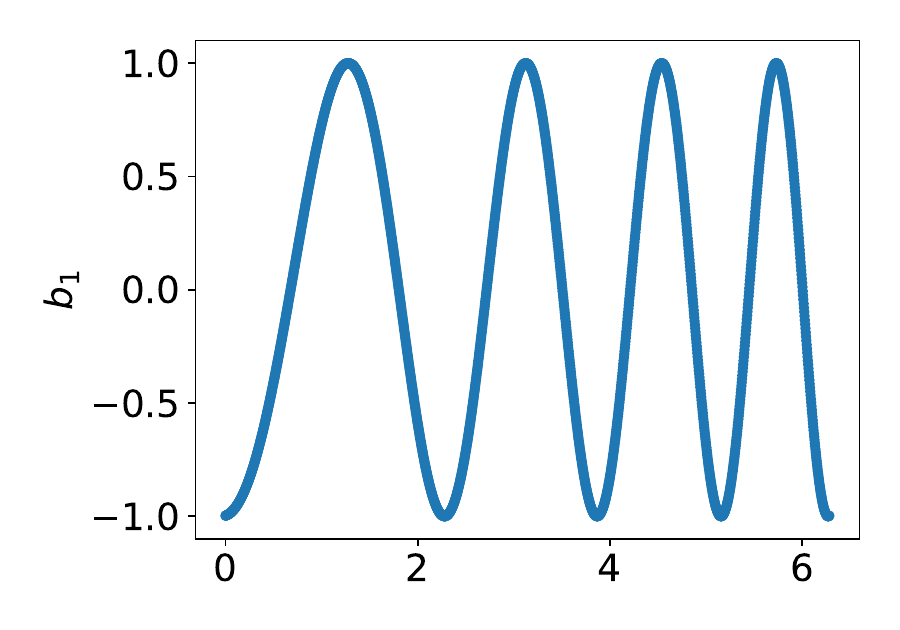}
      \subcaption{\label{fig:ex_1FIF_base_signal_1}}
    \end{subfigure}
  }
  \caption{In the top row, the signals sampled on a graph representing a grid of random points on the interval $[0, 2\pi)$ used for the GFT-IF and DB-IF algorithms.
    In particular, the signal $s$ (\subref{fig:ex_1signal}) is obtained as the sum of a high frequency component (\subref{fig:ex_1base_signal_0}) and a low frequency one (\subref{fig:ex_1base_signal_1}).
    In the bottom row, the same signals are sampled on an equispaced grid used for the classical IF algorithm.
  \label{fig:ex_1signals}}
\end{figure}

\begin{figure}[h]
  \makebox[\textwidth][c]{
    \begin{subfigure}{0.48\textwidth}
      \centering
      \includegraphics[width=\textwidth]{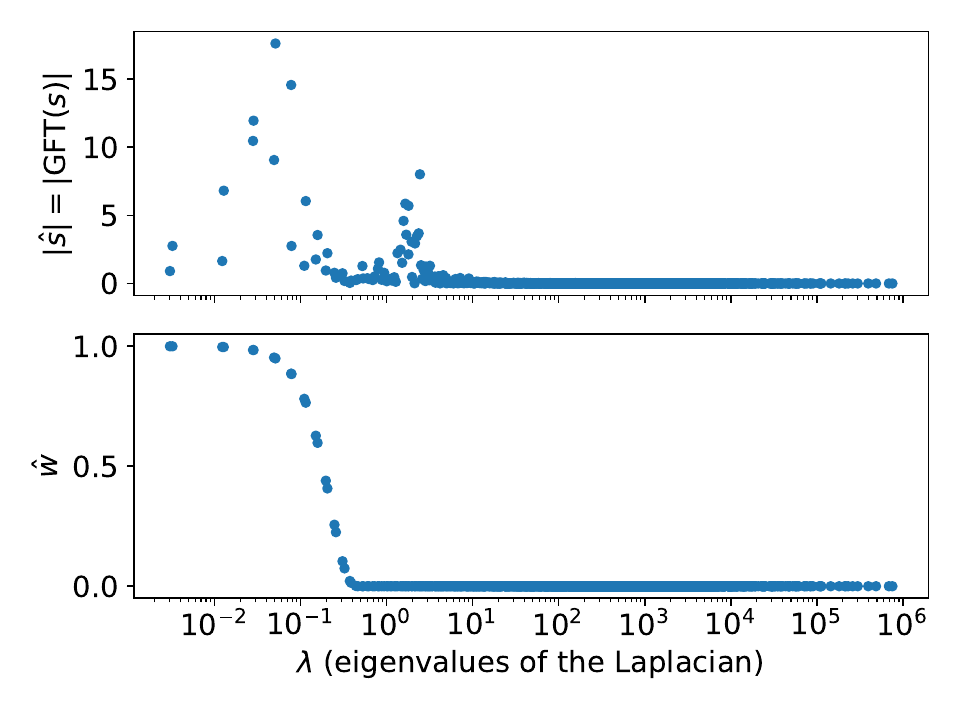}
      \caption{GFT of the signal $s$ (see \Cref{fig:ex_1signal}) compared with the convolution kernel used in the GFT-IF algorithm.\label{fig:ex_1GFT_kernel}}
    \end{subfigure}
    \ \
    \begin{subfigure}{0.48\textwidth}
      \centering
      \includegraphics[width=\textwidth]{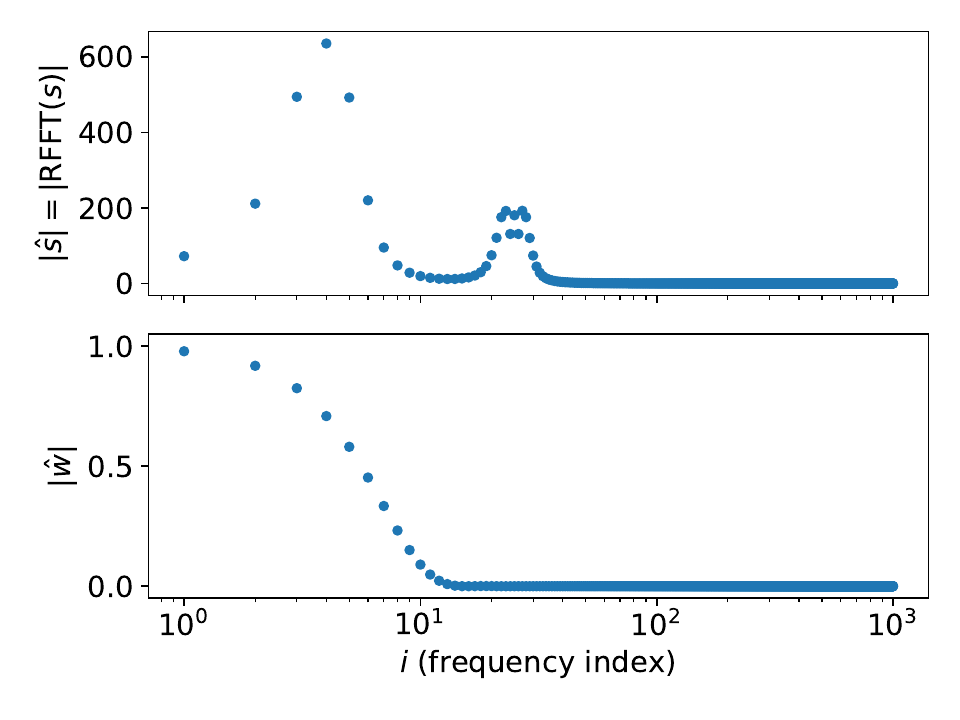}
      \caption{RFFT of the signal $s$ (see \Cref{fig:ex_1FIF_signal}) compared with the convolution kernel used in the FIF algorithm.\label{fig:ex_1FIF_kernel}}
    \end{subfigure}
  }
  \caption{\label{fig:ex_1kernels}}
\end{figure}

\begin{figure}[h]
  \makebox[\textwidth][c]{
    \begin{subfigure}{0.33\textwidth}
      \centering
      \includegraphics[width=\textwidth]{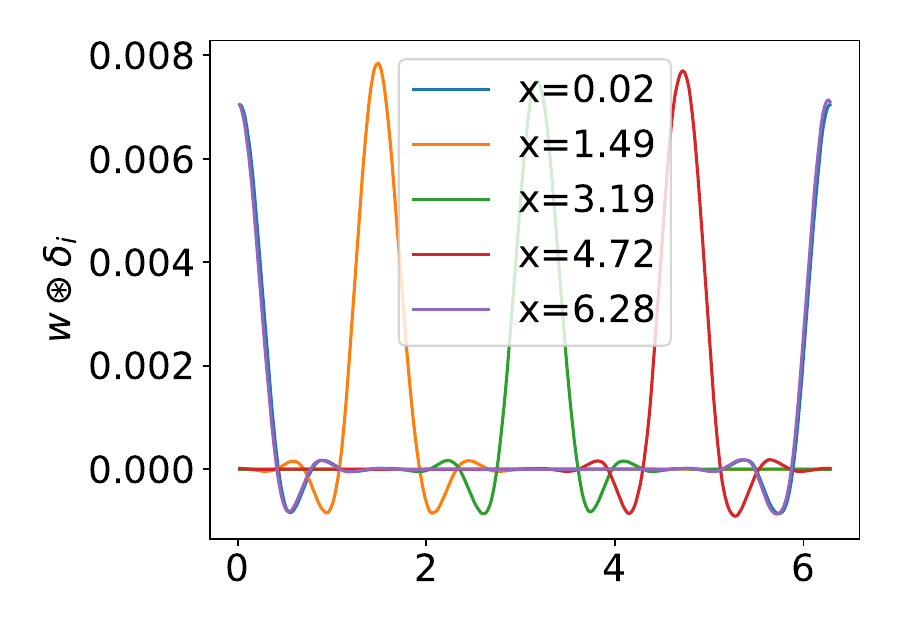}
      \caption{GFT-IF windows.\label{fig:ex_1GFT_windows}}
    \end{subfigure}
    \begin{subfigure}{0.33\textwidth}
      \centering
      \includegraphics[width=\textwidth]{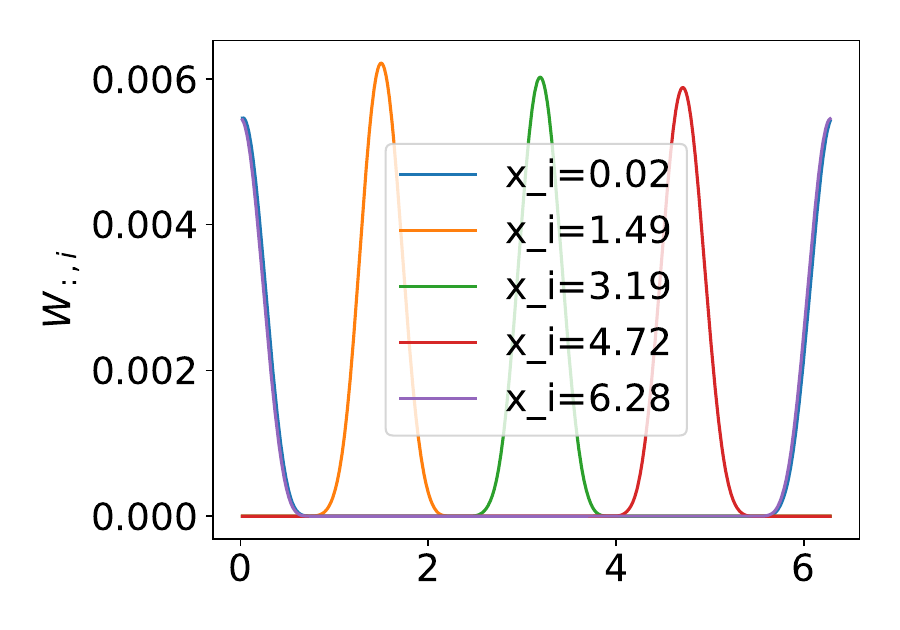}
      \caption{DB-IF windows.\label{fig:ex_1DB_windows}}
    \end{subfigure}
    \begin{subfigure}{0.33\textwidth}
      \centering
      \includegraphics[width=\textwidth]{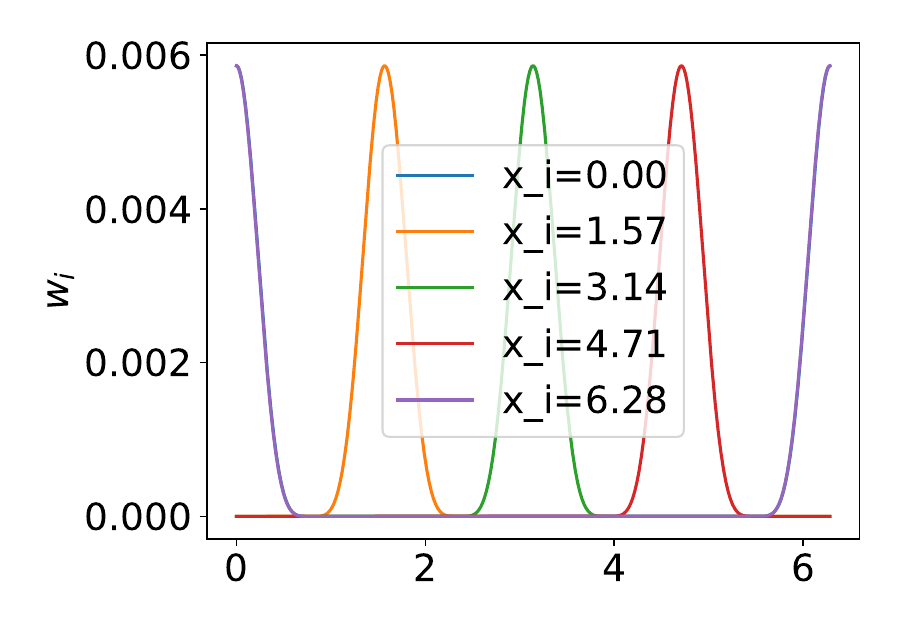}
      \caption{FIF windows.\label{fig:ex_1FIF_windows}}
    \end{subfigure}
  }
  \caption{Examples of the window functions used in the GFT-IF (\subref{fig:ex_1GFT_windows}), DB-IF (\subref{fig:ex_1DB_windows}) and FIF (\subref{fig:ex_1FIF_windows}) algorithms. In the case of the GFT-IF and FIF algorithms, those windows are obtained by convolving the kernel $w$ with delta signals centered at different vertices/points of the graph/domain. In the case of the DB-IF algorithm, those windows correspond to rows of the window matrix $W$.
  \label{fig:ex_1windows}}
\end{figure}

\begin{figure}[h]
  \centering
  \makebox[\textwidth][c]{
    \begin{subfigure}{0.33\textwidth}
      \centering
      \includegraphics[width=\textwidth]{./example_1_base_signal_0.pdf}
      \caption{\label{fig:ex_1comp_signal_0}}
    \end{subfigure}
    \begin{subfigure}{0.33\textwidth}
      \centering
      \includegraphics[width=\textwidth]{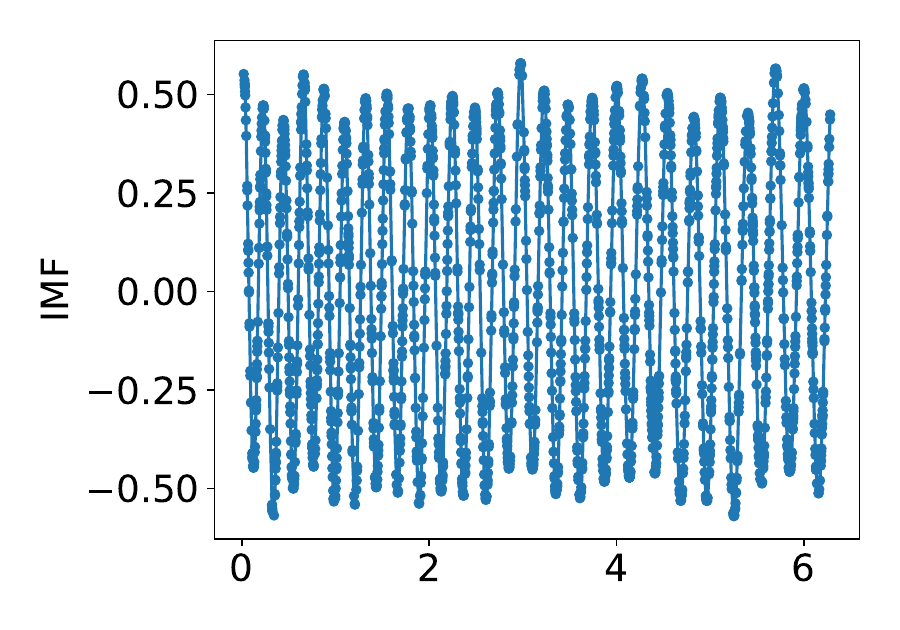}
      \caption{\label{fig:ex_1IMF_0}}
    \end{subfigure}
    \begin{subfigure}{0.33\textwidth}
      \centering
      \includegraphics[width=\textwidth]{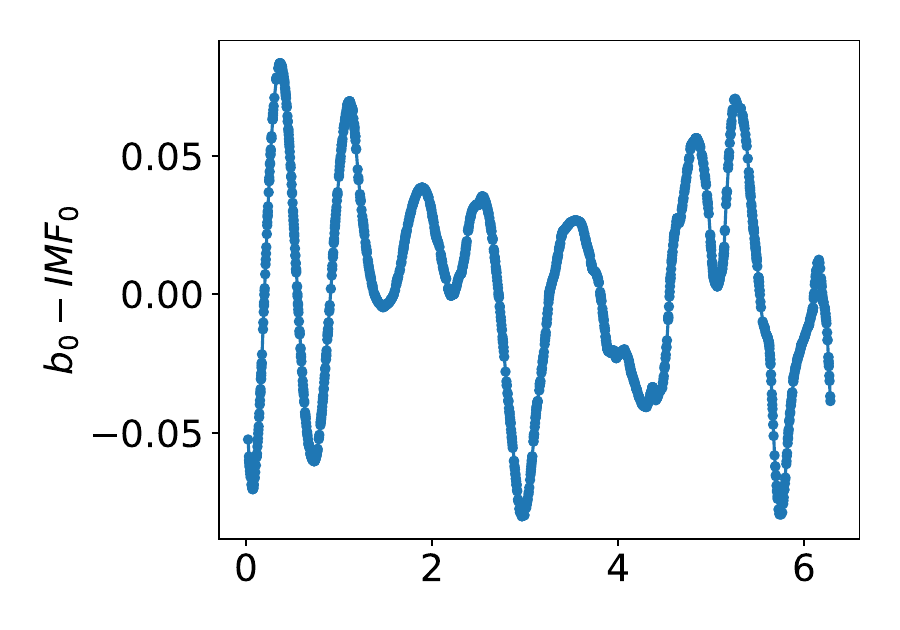}
      \caption{\label{fig:ex_1error_0}}
    \end{subfigure}
  }
  \makebox[\textwidth][c]{
    \begin{subfigure}{0.33\textwidth}
      \centering
      \includegraphics[width=\textwidth]{./example_1_base_signal_1.pdf}
      \caption{\label{fig:ex_1comp_signal_1}}
    \end{subfigure}
    \begin{subfigure}{0.33\textwidth}
      \centering
      \includegraphics[width=\textwidth]{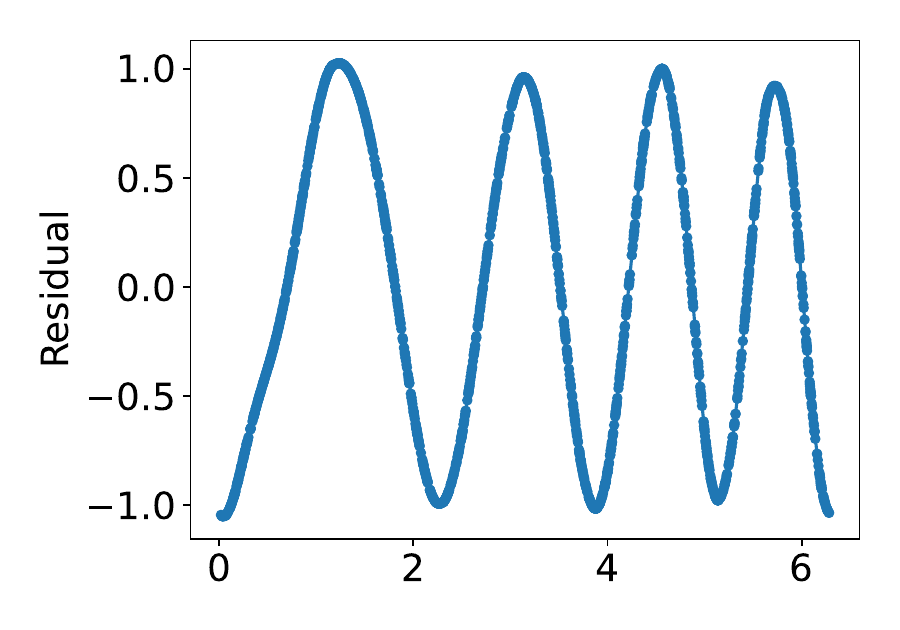}
      \caption{\label{fig:ex_1IMF_1}}
    \end{subfigure}
    \begin{subfigure}{0.33\textwidth}
      \centering
      \includegraphics[width=\textwidth]{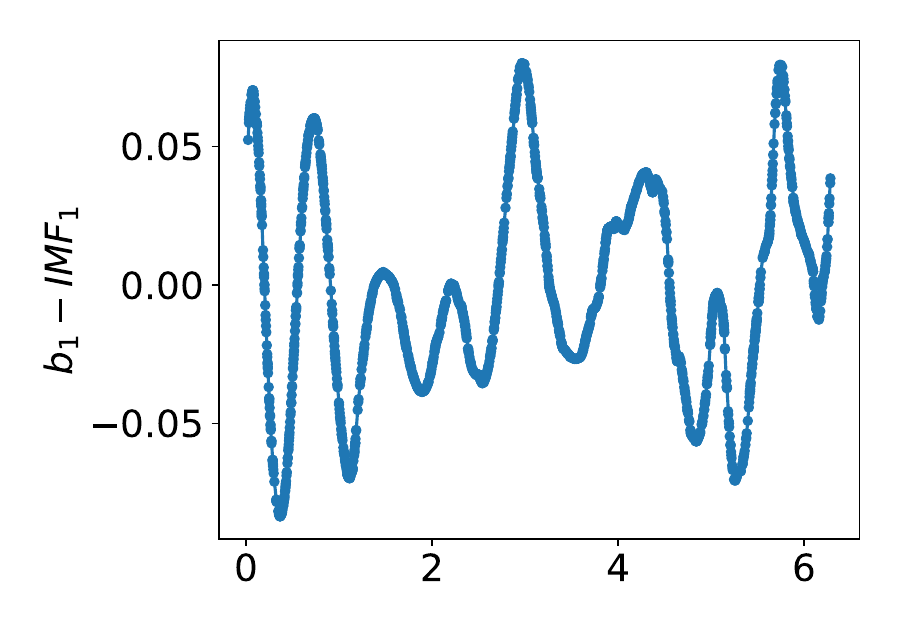}
      \caption{\label{fig:ex_1error_1}}
    \end{subfigure}
  }
  \caption{Results of the GFT-IF algorithm applied to the signal $s$ (shown in~\Cref{fig:ex_1signal}). The figure~(\subref{fig:ex_1IMF_0}) shows the first IMF obtained
    with the GFT-IF algorithm alongside the expected result
    (\subref{fig:ex_1comp_signal_0}) and their difference~(\subref{fig:ex_1error_0}). Similarly, the figure~(\subref{fig:ex_1IMF_1}) shows
  the residual signal compared to the expected result~(\subref{fig:ex_1comp_signal_1}) and their difference~(\subref{fig:ex_1error_1}).\label{fig:ex_1GFT_results}}
\end{figure}

\begin{figure}[h]
  \centering
  \makebox[\textwidth][c]{
    \begin{subfigure}{0.33\textwidth}
      \centering
      \includegraphics[width=\textwidth]{./example_1_base_signal_0.pdf}
      \caption{\label{fig:ex_1comp_DB_signal_0}}
    \end{subfigure}
    \begin{subfigure}{0.33\textwidth}
      \centering
      \includegraphics[width=\textwidth]{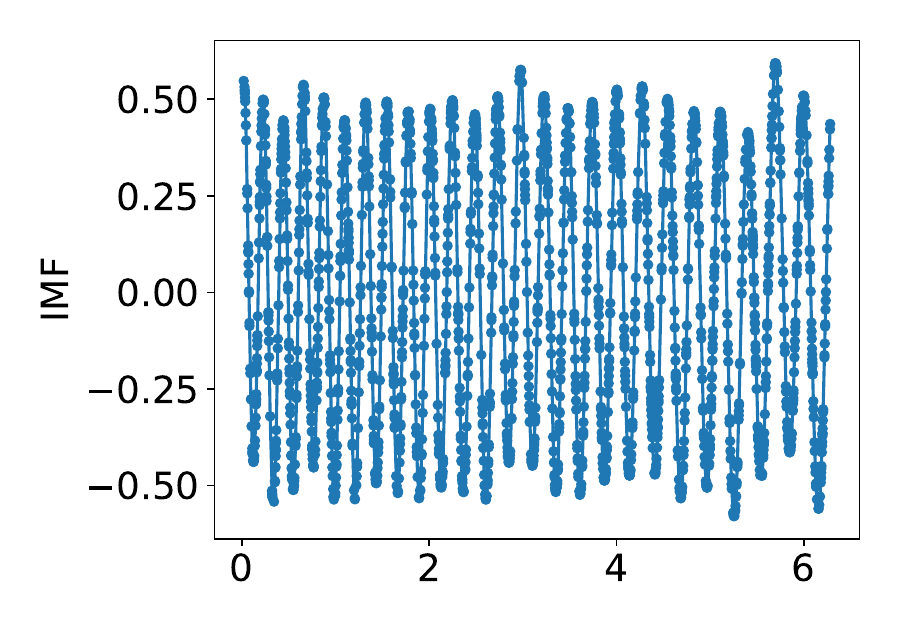}
      \caption{\label{fig:ex_1DB_IMF_0}}
    \end{subfigure}
    \begin{subfigure}{0.33\textwidth}
      \centering
      \includegraphics[width=\textwidth]{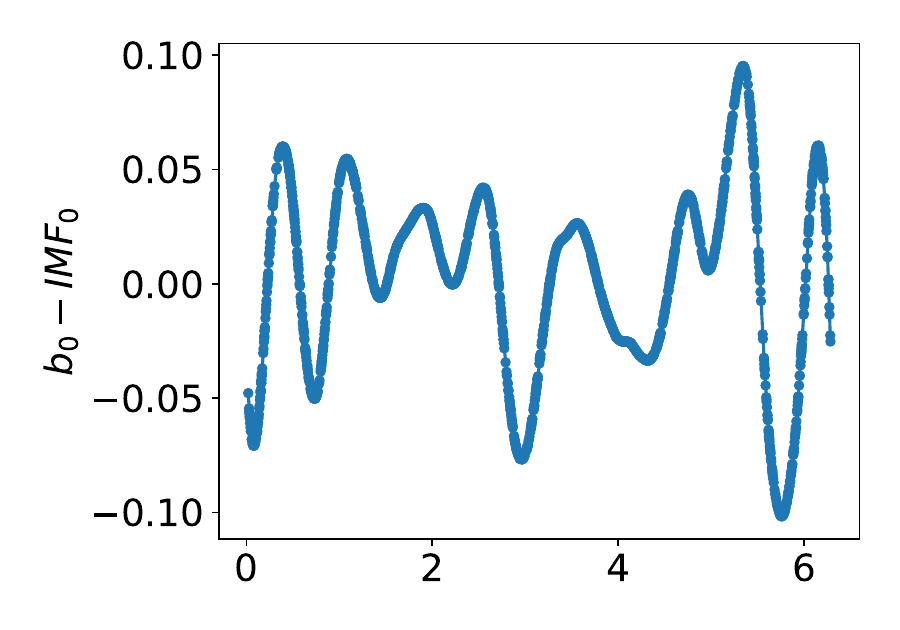}
      \caption{\label{fig:ex_1DB_error_0}}
    \end{subfigure}
  }
  \makebox[\textwidth][c]{
    \begin{subfigure}{0.33\textwidth}
      \centering
      \includegraphics[width=\textwidth]{./example_1_base_signal_1.pdf}
      \caption{\label{fig:ex_1comp_DB_signal_1}}
    \end{subfigure}
    \begin{subfigure}{0.33\textwidth}
      \centering
      \includegraphics[width=\textwidth]{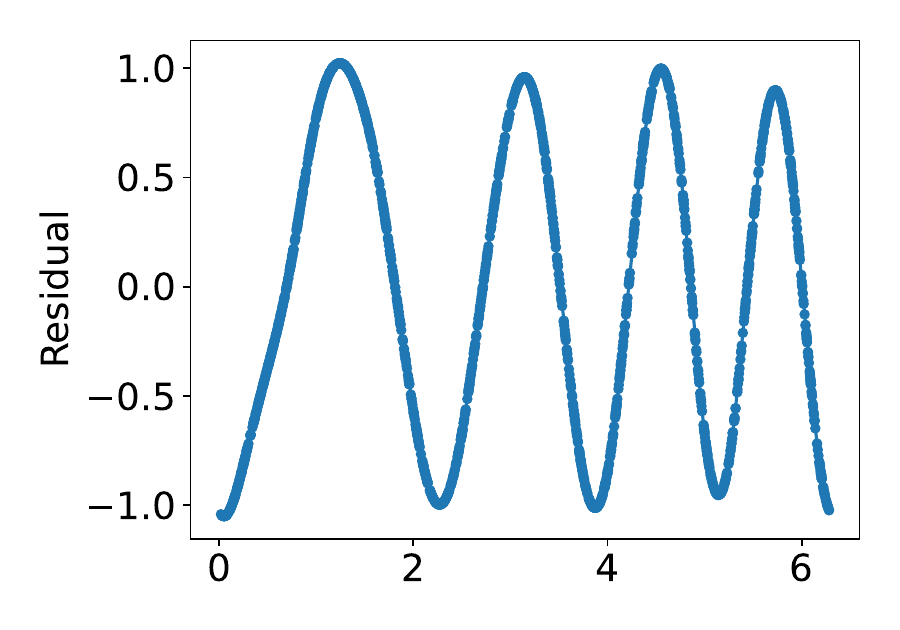}
      \caption{\label{fig:ex_1DB_IMF_1}}
    \end{subfigure}
    \begin{subfigure}{0.33\textwidth}
      \centering
      \includegraphics[width=\textwidth]{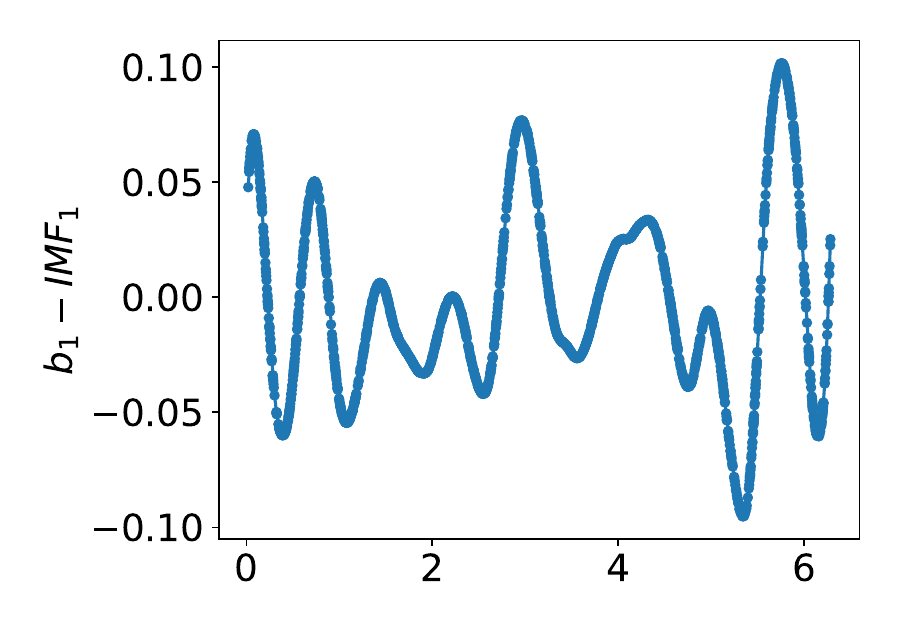}
      \caption{\label{fig:ex_1DB_error_1}}
    \end{subfigure}
  }
  \caption{Results of the DB-IF algorithm applied to the signal $s$ (shown in~\Cref{fig:ex_1signal}). The figure~(\subref{fig:ex_1DB_IMF_0}) shows the first IMF obtained
    with the DB-IF algorithm alongside the expected result
    (\subref{fig:ex_1comp_DB_signal_0}) and their difference~(\subref{fig:ex_1DB_error_0}). Similarly, the figure~(\subref{fig:ex_1DB_IMF_1}) shows
  the residual signal compared to the expected result~(\subref{fig:ex_1comp_DB_signal_1}) and their difference~(\subref{fig:ex_1DB_error_1}).\label{fig:ex_1DB_results}}
\end{figure}

\begin{figure}[h]
  \centering
  \makebox[\textwidth][c]{
    \begin{subfigure}{0.33\textwidth}
      \centering
      \includegraphics[width=\textwidth]{./example_1_FIF_base_signal_0.pdf}
      \caption{\label{fig:ex_1comp_FIF_signal_0}}
    \end{subfigure}
    \begin{subfigure}{0.33\textwidth}
      \centering
      \includegraphics[width=\textwidth]{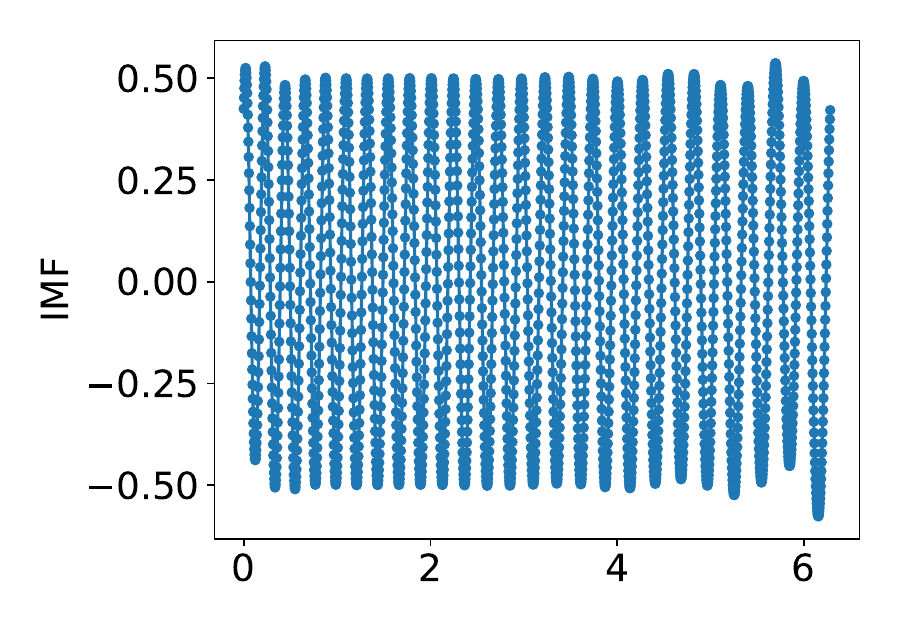}
      \caption{\label{fig:ex_1FIF_IMF_0}}
    \end{subfigure}
    \begin{subfigure}{0.33\textwidth}
      \centering
      \includegraphics[width=\textwidth]{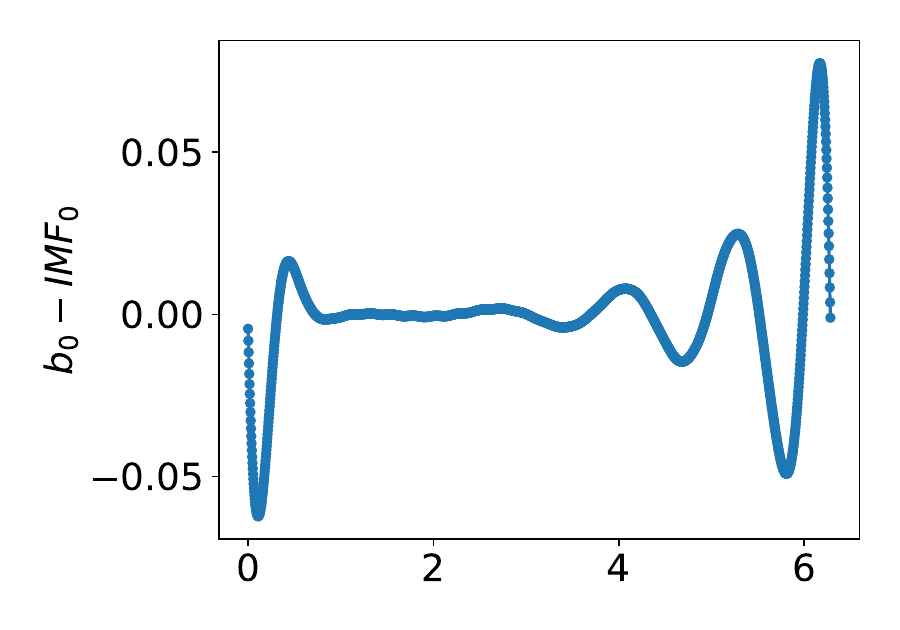}
      \caption{\label{fig:ex_1FIF_error_0}}
    \end{subfigure}
  }
  \makebox[\textwidth][c]{
    \begin{subfigure}{0.33\textwidth}
      \centering
      \includegraphics[width=\textwidth]{./example_1_FIF_base_signal_1.pdf}
      \caption{\label{fig:ex_1comp_FIF_signal_1}}
    \end{subfigure}
    \begin{subfigure}{0.33\textwidth}
      \centering
      \includegraphics[width=\textwidth]{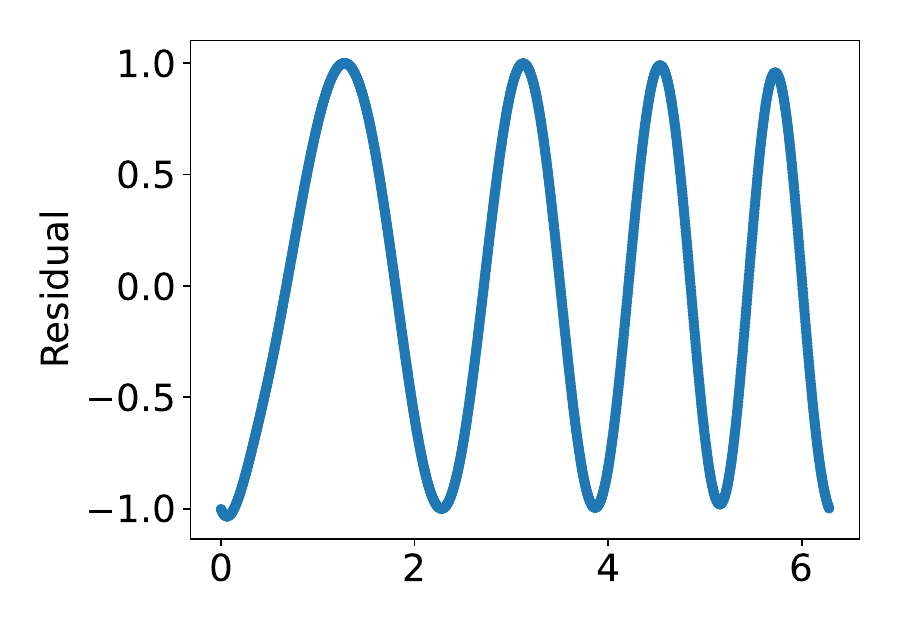}
      \caption{\label{fig:ex_1FIF_IMF_1}}
    \end{subfigure}
    \begin{subfigure}{0.33\textwidth}
      \centering
      \includegraphics[width=\textwidth]{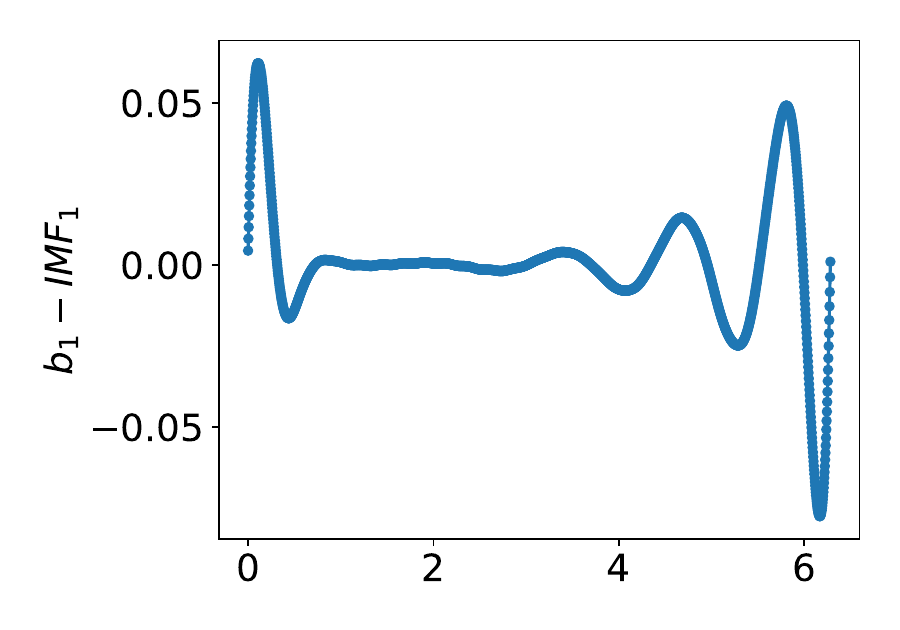}
      \caption{\label{fig:ex_1FIF_error_1}}
    \end{subfigure}
  }
  \caption{Results of the FIF algorithm applied to the signal $s$ (shown in~\Cref{fig:ex_1FIF_signal}). The figure~(\subref{fig:ex_1FIF_IMF_0}) shows the first IMF obtained
    with the FIF algorithm alongside the expected result
    (\subref{fig:ex_1comp_FIF_signal_0}) and their difference~(\subref{fig:ex_1FIF_error_0}). Similarly, the figure~(\subref{fig:ex_1FIF_IMF_1}) shows
  the residual signal compared to the expected result~(\subref{fig:ex_1comp_FIF_signal_1}) and their difference~(\subref{fig:ex_1FIF_error_1}).\label{fig:ex_1FIF_results}}
\end{figure}

\FloatBarrier

\begin{table}[width=.4\linewidth,cols=3,pos=h]
  \caption{Precomputation CPU time (seconds) for different numbers of points $n$.}
  \label{tab:ex_1CPU_time_precomputation}
  \begin{tabular*}{\tblwidth}{@{}LLL@{}}
    \toprule
    $n$ & GFT-IF & DB-IF \\ [0.5ex]
    \midrule
    128 & 7.32e-03 & 3.07e-05 \\
    512 & 1.40e-01 & 2.37e-03 \\
    2048 & 5.96 & 2.25e-02 \\
    8192 & 3.81e+02 & 2.87e-01 \\
    \bottomrule
  \end{tabular*}
\end{table}

\begin{table}[width=.4\linewidth,cols=4,pos=h]
  \centering
  \caption{Execution CPU time (seconds) excluding precomputations ($k=10$ IMFs, $m=10$ iterations).}
  \label{tab:ex_1CPU_time_total_ex_precomputation}
  \begin{tabular*}{\tblwidth}{@{}LLLL@{}}
    \toprule
    $n$ & GFT-IF & DB-IF & FIF \\ [0.5ex]
    \midrule
    128 &	3.54e-04 &	2.12e-03 &	5.30e-04 \\
    512 &	1.71e-03 &	5.83e-02 &	5.96e-04 \\
    2048 &	2.24e-02 &	9.15e-01 &	8.91e-04 \\
    8192 &	2.95e-01 &	1.90e+01 &	1.99e-03 \\
    \bottomrule
  \end{tabular*}
\end{table}

\subsection{Example 2: Delaunay triangulation of random points in 2D}
\label{sec:example_2}
In this example, we have tested the GFT-IF and the DB-IF algorithms on a 2D signal on a grid of random points uniformly distributed on the square $[0, 2\pi) \times [0, 2\pi)$.
The graph is constructed as the Delaunay triangulation of those points \cite{Delaunay}. The edges are weighted by $1/d(i,j)$, where $d(i,j)$ is the Euclidean distance between the vertices $i$ and $j$.
The tested signal is the sum of two sinusoidal functions with non-stationary frequencies
\begin{equation*}
  \begin{split}
    b_0(x,y) &= \frac{1}{2} \sin(5 x + 5 y),\\
    b_1(x,y) &= \cos(x - y),\\
  \end{split}
\end{equation*}
and it is shown in \Cref{fig:ex_2signals}.

\Cref{fig:ex_2GFT_kernel} shows the GFT of
the signal $s$ compared with the convolution kernel used in the GFT-IF
algorithm.
\Cref{fig:ex_2windows} shows some examples of the window functions used in the GFT-IF and DB-IF methods.
We observe that the windows of both methods are qualitatively similar, but the windows of the GFT-IF algorithm have some small negative values, while the windows of the DB-IF algorithm are slightly shifted towards the center of the domain.

\Cref{fig:ex_2GFT_results,fig:ex_2DB_results} shows the results of the GFT-IF and DB-IF algorithms, respectively.
In both cases, the methods are able to extract IMFs that are qualitatively similar to the expected components of the signal, but the error of the decomposition is not negligible due to some low-frequency noise and boundary effects showing up in the error, which are cancelled when the two components are summed together.

\begin{figure}[h]
  \centering
  \makebox[\textwidth][c]{
    \begin{subfigure}{0.33\textwidth}
      \centering
      \includegraphics[width=\textwidth]{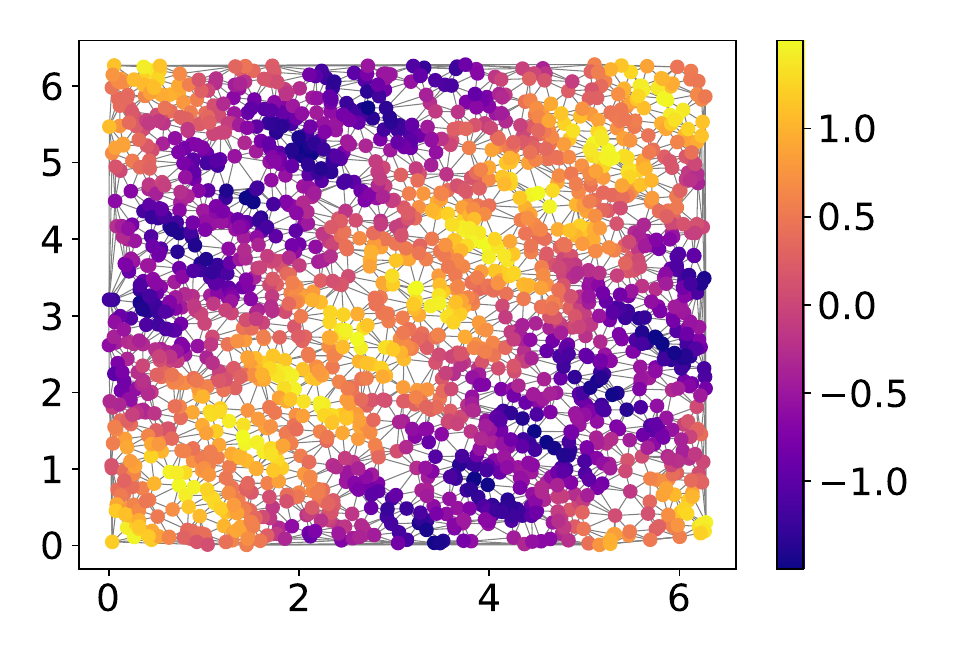}
      \subcaption{\label{fig:ex_2signal}}
    \end{subfigure}
    \begin{subfigure}{0.33\textwidth}
      \centering
      \includegraphics[width=\textwidth]{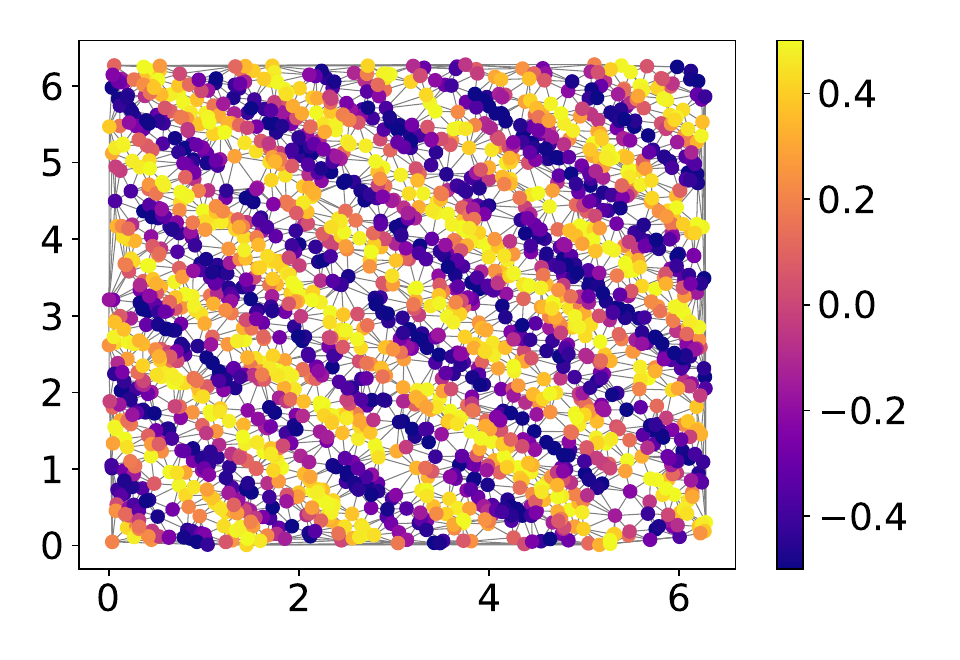}
      \subcaption{\label{fig:ex_2base_signal_0}}
    \end{subfigure}
    \begin{subfigure}{0.33\textwidth}
      \centering
      \includegraphics[width=\textwidth]{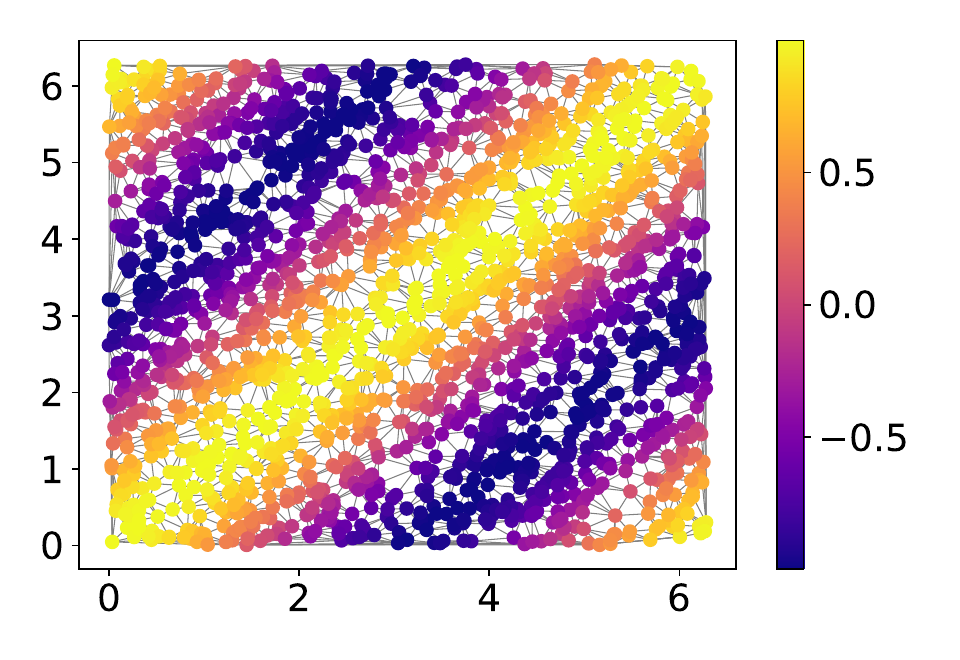}
      \subcaption{\label{fig:ex_2base_signal_1}}
    \end{subfigure}
  }
  \caption{(\subref{fig:ex_2signal}) Signal s obtained as the sum of a high frequency component (\subref{fig:ex_2base_signal_0}) and a low frequency one (\subref{fig:ex_2base_signal_1}).
  \label{fig:ex_2signals}}
\end{figure}

\begin{figure}[h]
  \centering
    \includegraphics[width=.6\textwidth]{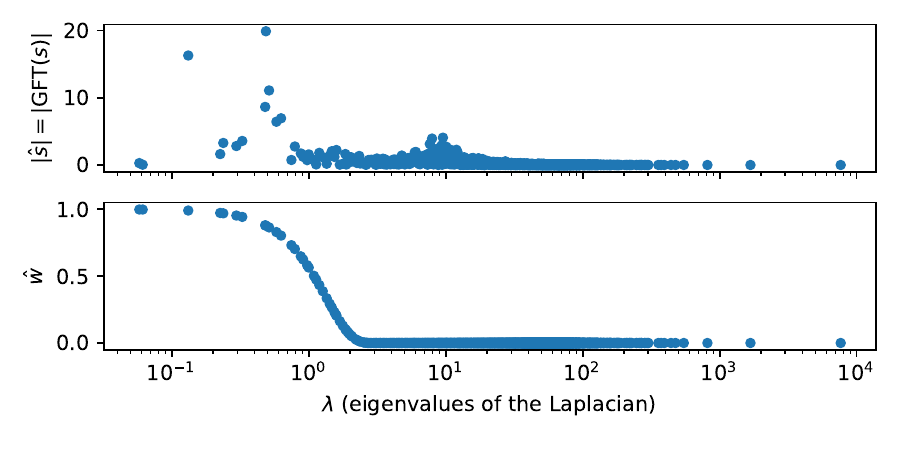}
    \caption{GFT of the signal $s$ (see \Cref{fig:ex_2signal}) compared with the convolution kernel used in the GFT-IF algorithm.\label{fig:ex_2GFT_kernel}}
\end{figure}

\begin{figure}[h]
  \makebox[\textwidth][c]{
    \begin{subfigure}{0.33\textwidth}
      \centering
      \includegraphics[width=\textwidth]{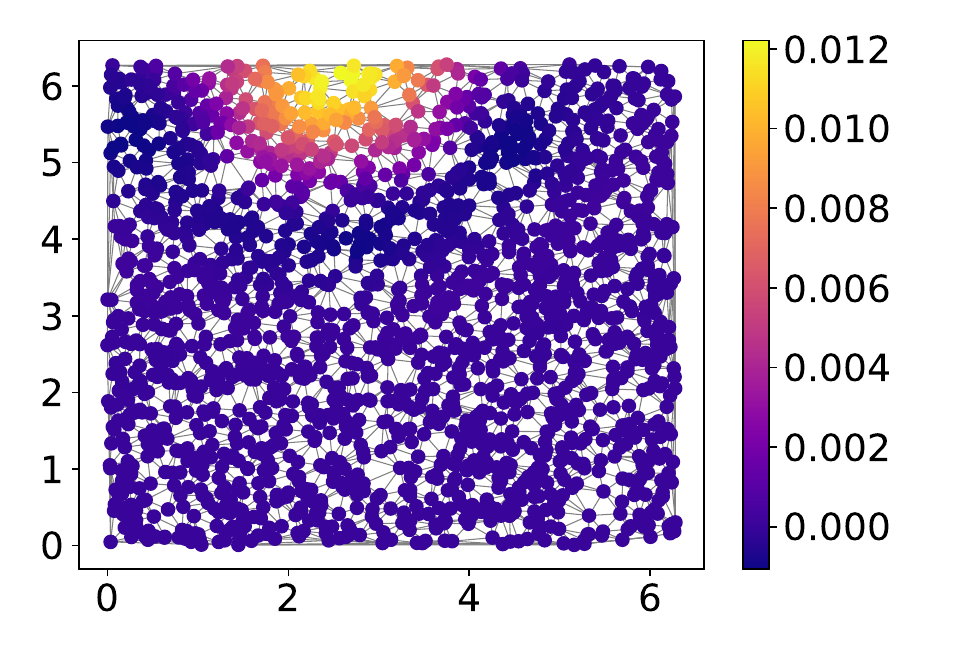}
    \end{subfigure}
    \begin{subfigure}{0.33\textwidth}
      \centering
      \includegraphics[width=\textwidth]{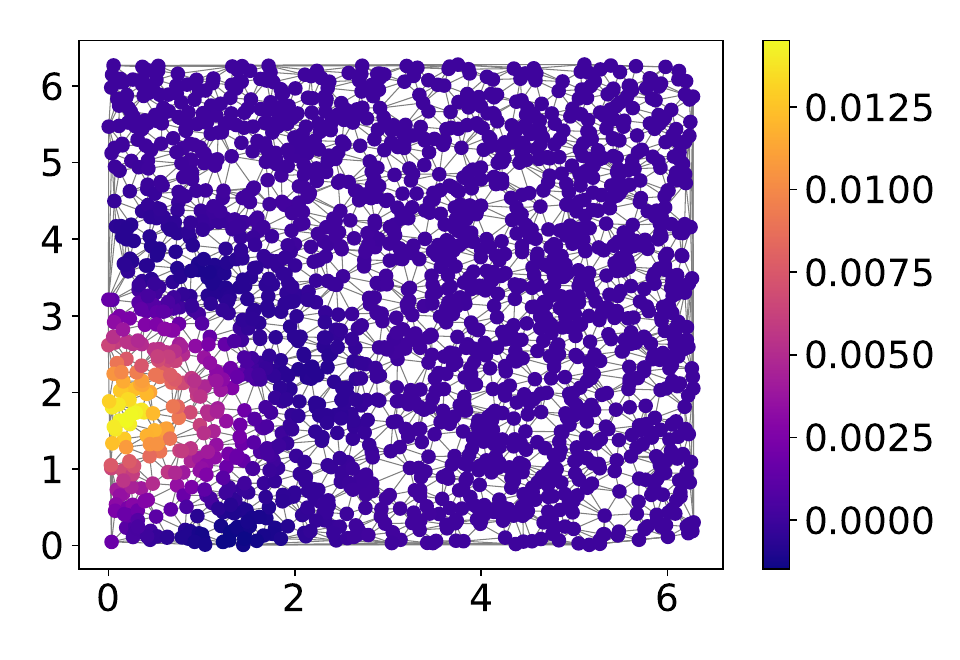}
    \end{subfigure}
    \begin{subfigure}{0.33\textwidth}
      \centering
      \includegraphics[width=\textwidth]{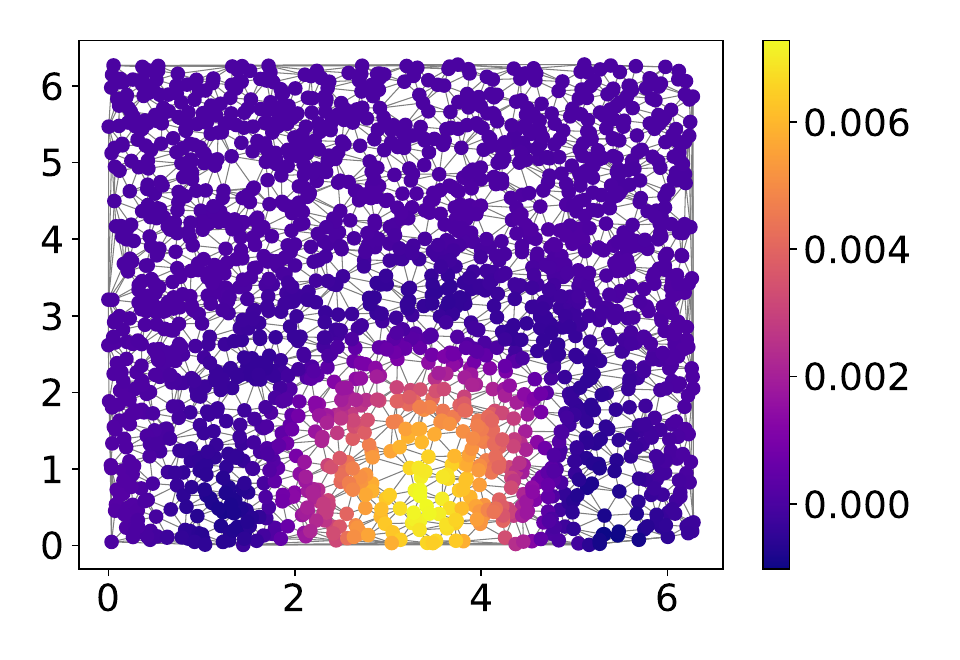}
    \end{subfigure}
  }
  \makebox[\textwidth][c]{
    \begin{subfigure}{0.33\textwidth}
      \centering
      \includegraphics[width=\textwidth]{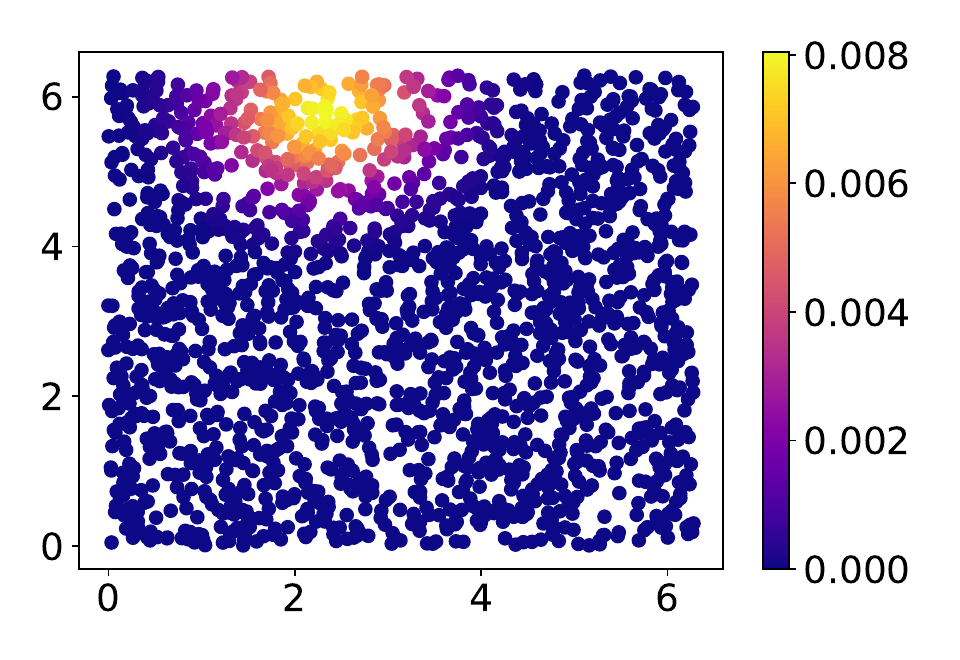}
    \end{subfigure}
    \begin{subfigure}{0.33\textwidth}
      \centering
      \includegraphics[width=\textwidth]{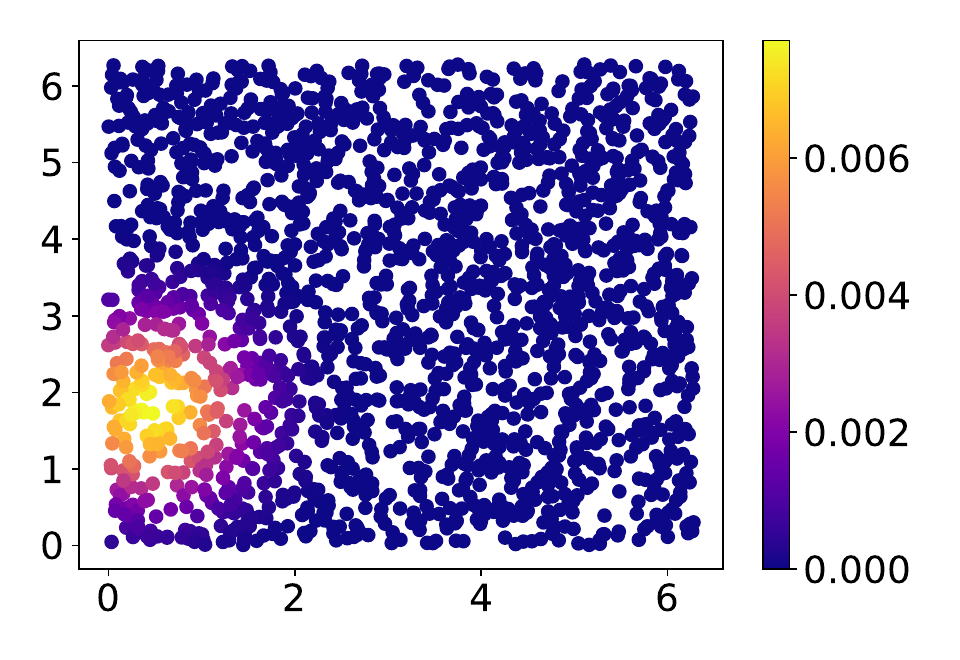}
    \end{subfigure}
    \begin{subfigure}{0.33\textwidth}
      \centering
      \includegraphics[width=\textwidth]{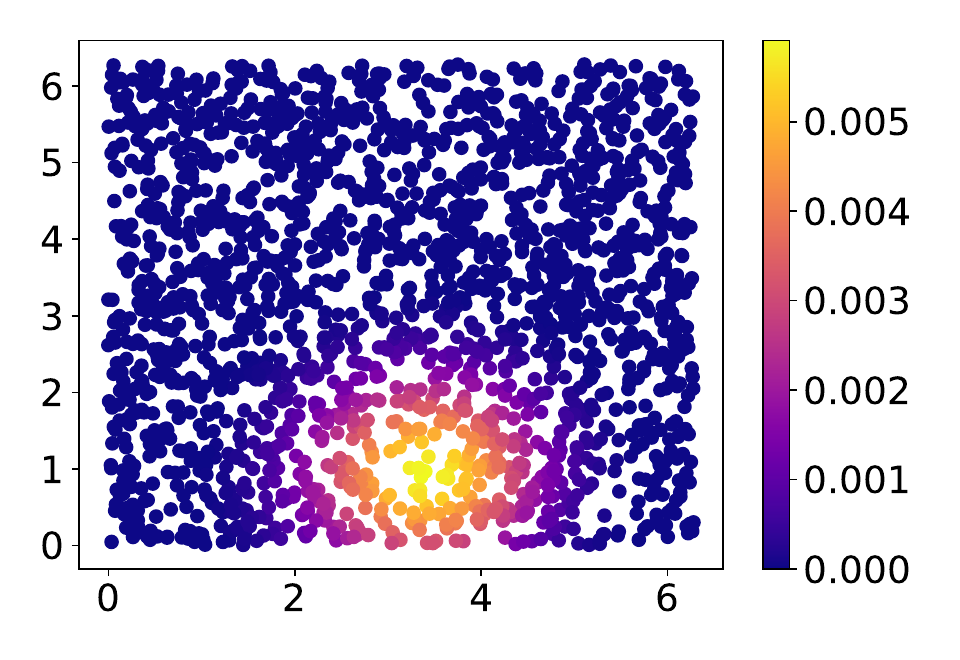}
    \end{subfigure}
  }
  \caption{Examples of the window functions used in the GFT-IF (first row) and DB-IF (second row) algorithms. In the case of the GFT-IF algorithm, those windows are obtained by convolving the kernel $w$ with delta signals centered at different vertices of the graph. In the case of the DB-IF algorithm, those windows correspond to rows of the window matrix $W$.
  \label{fig:ex_2windows}}
\end{figure}

\begin{figure}[h]
  \centering
  \makebox[\textwidth][c]{
    \begin{subfigure}{0.33\textwidth}
      \centering
      \includegraphics[width=\textwidth]{./example_2_base_signal_0.pdf}
      \caption{\label{fig:ex_2comp_signal_0}}
    \end{subfigure}
    \begin{subfigure}{0.33\textwidth}
      \centering
      \includegraphics[width=\textwidth]{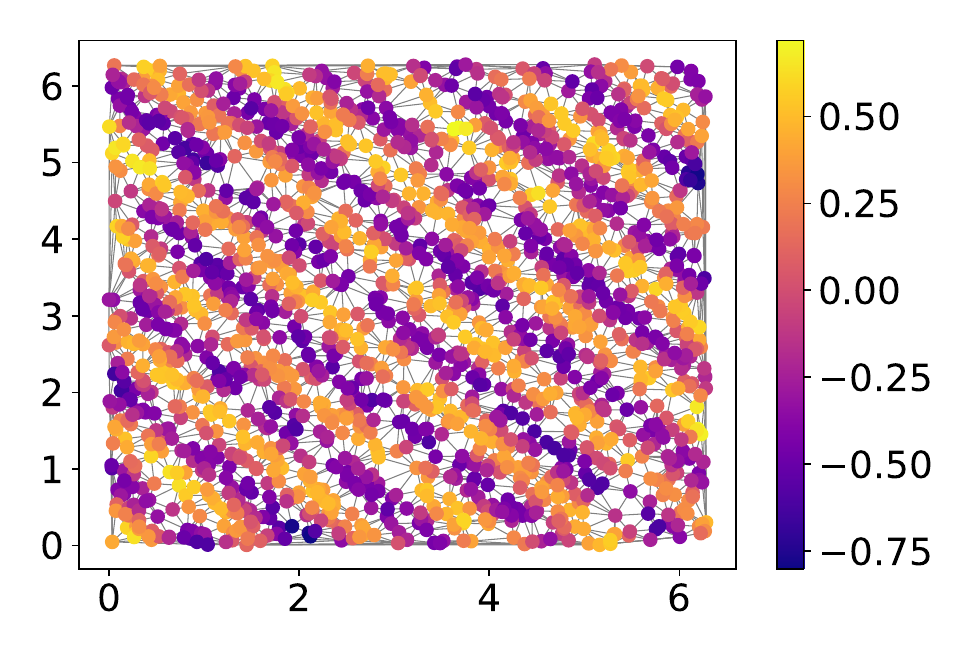}
      \caption{\label{fig:ex_2IMF_0}}
    \end{subfigure}
    \begin{subfigure}{0.33\textwidth}
      \centering
      \includegraphics[width=\textwidth]{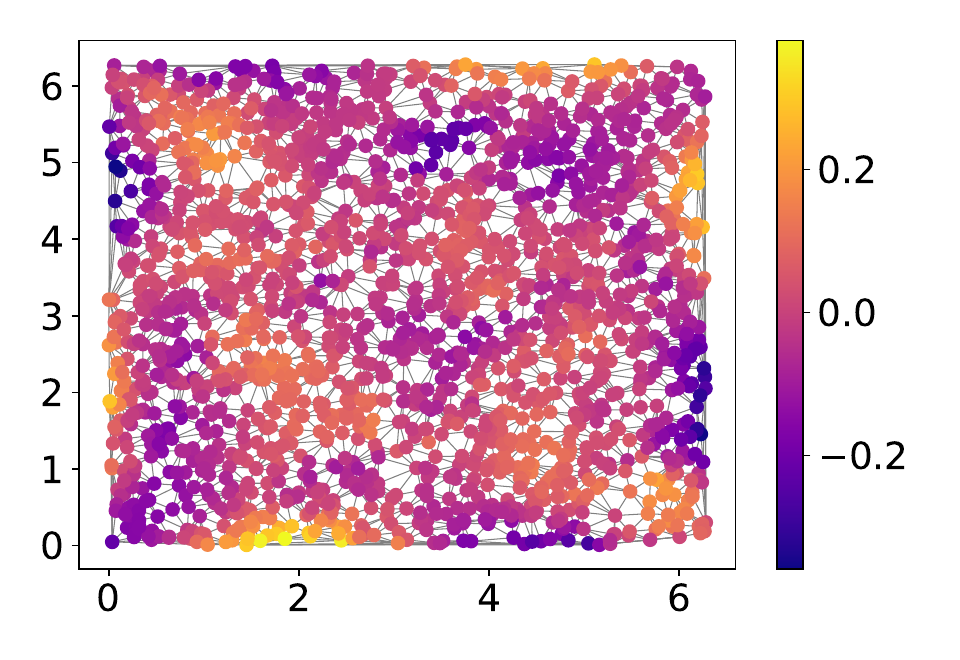}
      \caption{\label{fig:ex_2error_0}}
    \end{subfigure}
  }
  \makebox[\textwidth][c]{
    \begin{subfigure}{0.33\textwidth}
      \centering
      \includegraphics[width=\textwidth]{./example_2_base_signal_1.pdf}
      \caption{\label{fig:ex_2comp_signal_1}}
    \end{subfigure}
    \begin{subfigure}{0.33\textwidth}
      \centering
      \includegraphics[width=\textwidth]{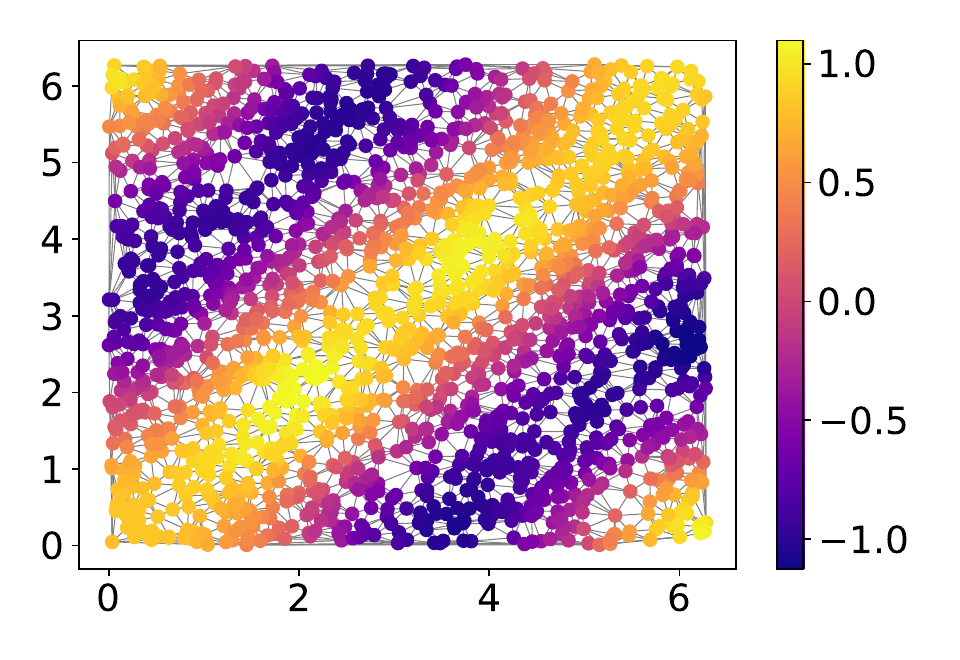}
      \caption{\label{fig:ex_2IMF_1}}
    \end{subfigure}
    \begin{subfigure}{0.33\textwidth}
      \centering
      \includegraphics[width=\textwidth]{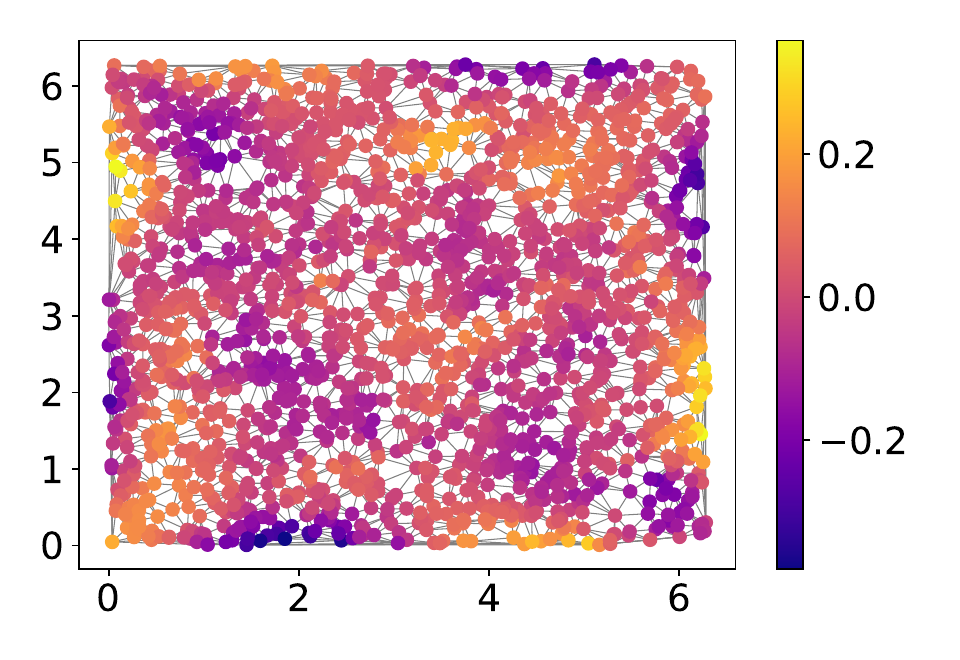}
      \caption{\label{fig:ex_2error_1}}
    \end{subfigure}
  }
  \caption{Results of the GFT-IF algorithm applied to the signal $s$ (shown in~\Cref{fig:ex_2signal}). The figure~(\subref{fig:ex_2IMF_0}) shows the first IMF obtained
    with the GFT-IF algorithm alongside the expected result
    (\subref{fig:ex_2comp_signal_0}) and their difference~(\subref{fig:ex_2error_0}). Similarly, the figure~(\subref{fig:ex_2IMF_1}) shows
  the residual signal compared to the expected result~(\subref{fig:ex_2comp_signal_1}) and their difference~(\subref{fig:ex_2error_1}).\label{fig:ex_2GFT_results}}
\end{figure}

\begin{figure}[h]
  \centering
  \makebox[\textwidth][c]{
    \begin{subfigure}{0.33\textwidth}
      \centering
      \includegraphics[width=\textwidth]{./example_2_base_signal_0.pdf}
      \caption{\label{fig:ex_2comp_DB_signal_0}}
    \end{subfigure}
    \begin{subfigure}{0.33\textwidth}
      \centering
      \includegraphics[width=\textwidth]{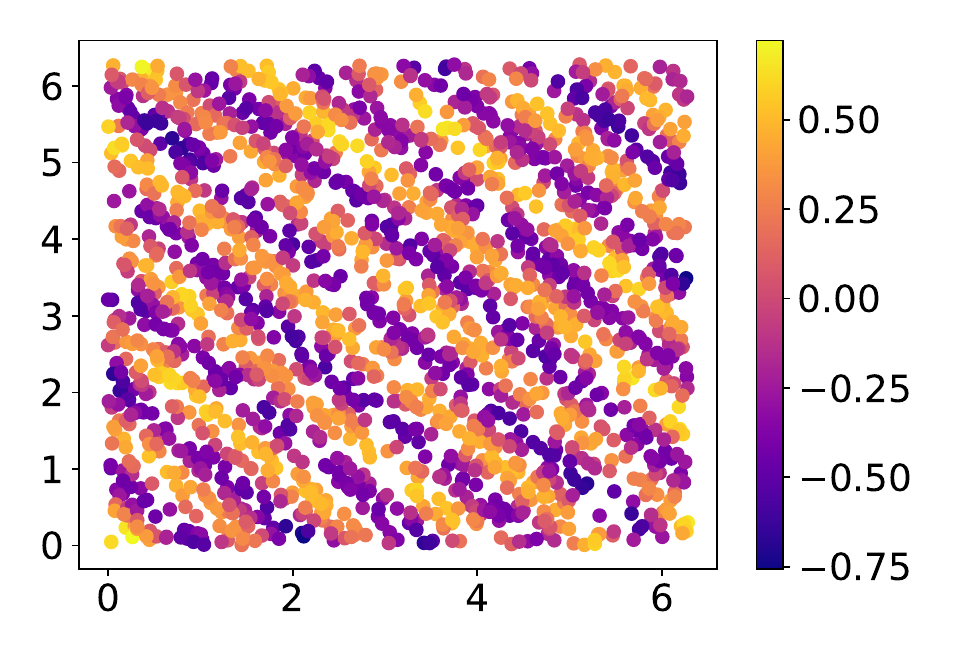}
      \caption{\label{fig:ex_2DB_IMF_0}}
    \end{subfigure}
    \begin{subfigure}{0.33\textwidth}
      \centering
      \includegraphics[width=\textwidth]{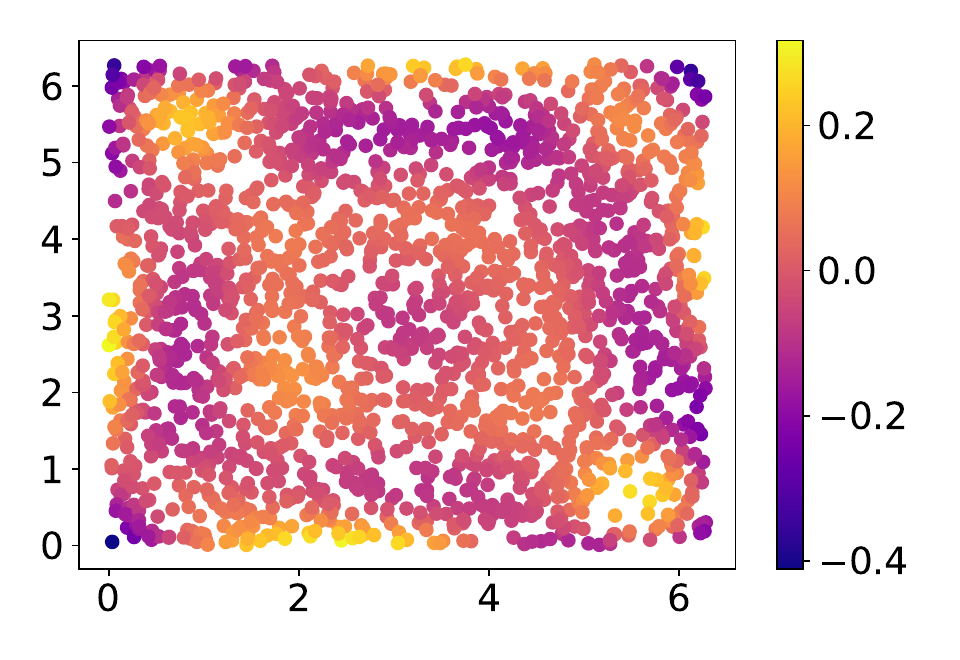}
      \caption{\label{fig:ex_2DB_error_0}}
    \end{subfigure}
  }
  \makebox[\textwidth][c]{
    \begin{subfigure}{0.33\textwidth}
      \centering
      \includegraphics[width=\textwidth]{./example_2_base_signal_1.pdf}
      \caption{\label{fig:ex_2comp_DB_signal_1}}
    \end{subfigure}
    \begin{subfigure}{0.33\textwidth}
      \centering
      \includegraphics[width=\textwidth]{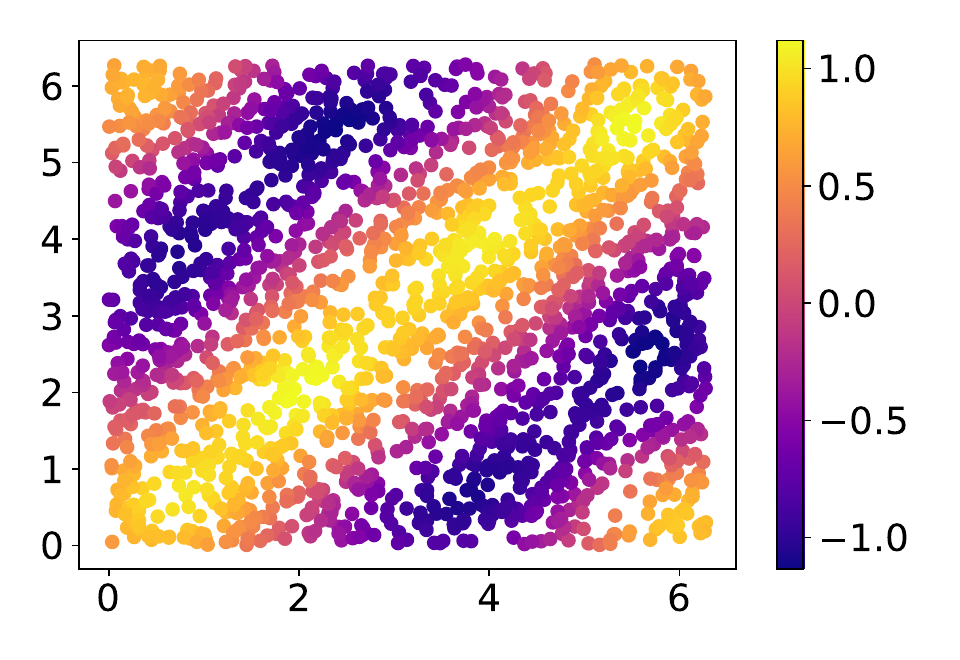}
      \caption{\label{fig:ex_2DB_IMF_1}}
    \end{subfigure}
    \begin{subfigure}{0.33\textwidth}
      \centering
      \includegraphics[width=\textwidth]{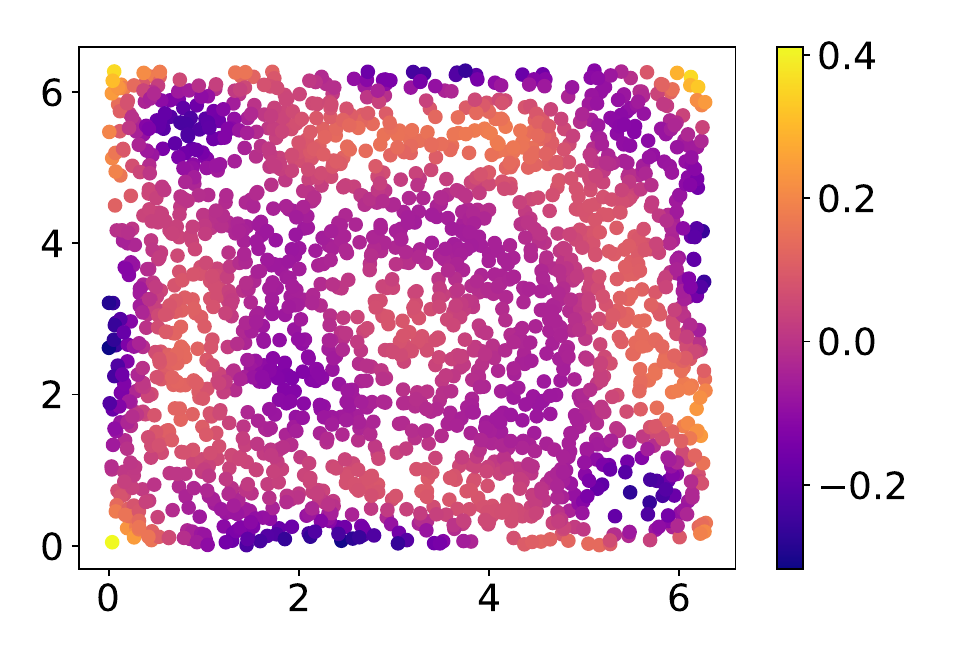}
      \caption{\label{fig:ex_2DB_error_1}}
    \end{subfigure}
  }
  \caption{Results of the DB-IF algorithm applied to the signal $s$ (shown in~\Cref{fig:ex_2signal}). The figure~(\subref{fig:ex_2DB_IMF_0}) shows the first IMF obtained
    with the DB-IF algorithm alongside the expected result
    (\subref{fig:ex_2comp_DB_signal_0}) and their difference~(\subref{fig:ex_2DB_error_0}). Similarly, the figure~(\subref{fig:ex_2DB_IMF_1}) shows
  the residual signal compared to the expected result~(\subref{fig:ex_2comp_DB_signal_1}) and their difference~(\subref{fig:ex_2DB_error_1}).\label{fig:ex_2DB_results}}
\end{figure}

\FloatBarrier

\subsection{Example 3: magnitude time series of the global seismicity}
\label{sec:example_3}

In this example, we analyze the magnitude time series of the global seismicity for events of magnitude $M \ge 5.0$ from January 1, 1976 to December 20, 1980. The data was downloaded from the Global Centroid Moment Tensor (CMT) Project (\url{https://www.globalcmt.org/}). Due to the aleatory nature of the seismicity, the series is non-uniformly sampled in time. This gives a good example of a case in which the classical FIF algorithm cannot be directly applied, while the GFT-IF and DB-IF algorithms can be applied directly on the real data.

The graph used in this case is a ring graph constructed in the same way as in Example 1 in \Cref{sec:example_1}.

\Cref{fig:ex_3IMFs} shows the original signal in the first row, and in the following rows the IMFs obtained with the GFT-IF and DB-IF algorithms.
Even though the results are difficult to interpret due to the noisy and unpredictable nature of the seismicity, we can observe that the extracted trend is similar for both methods, but, as the frequency of the IMFs increases, the results of the two methods start to differ more and more.
\Cref{fig:ex_3windows} shows some examples of the windows used for each IMF extraction in the GFT-IF and DB-IF methods.
For the GFT-IF algorithm, the support of the window kernel in the spectral domain was manually chosen to obtain windows similar to those of the DB-IF algorithm.
\Cref{fig:ex_3GFT_windows_and_IMFs_spectrum} shows the window kernels used for each IMF extraction in the GFT-IF algorithm and the spectrum of the obtained IMFs.
This figure gives a good idea of how IF methods work in the spectral domain, by doing a soft partition of the spectrum of the signal to extract different components of it.

\begin{figure}[h]
  \makebox[\textwidth][c]{
    \begin{subfigure}{0.48\textwidth}
        \centering
        \includegraphics[width=\textwidth]{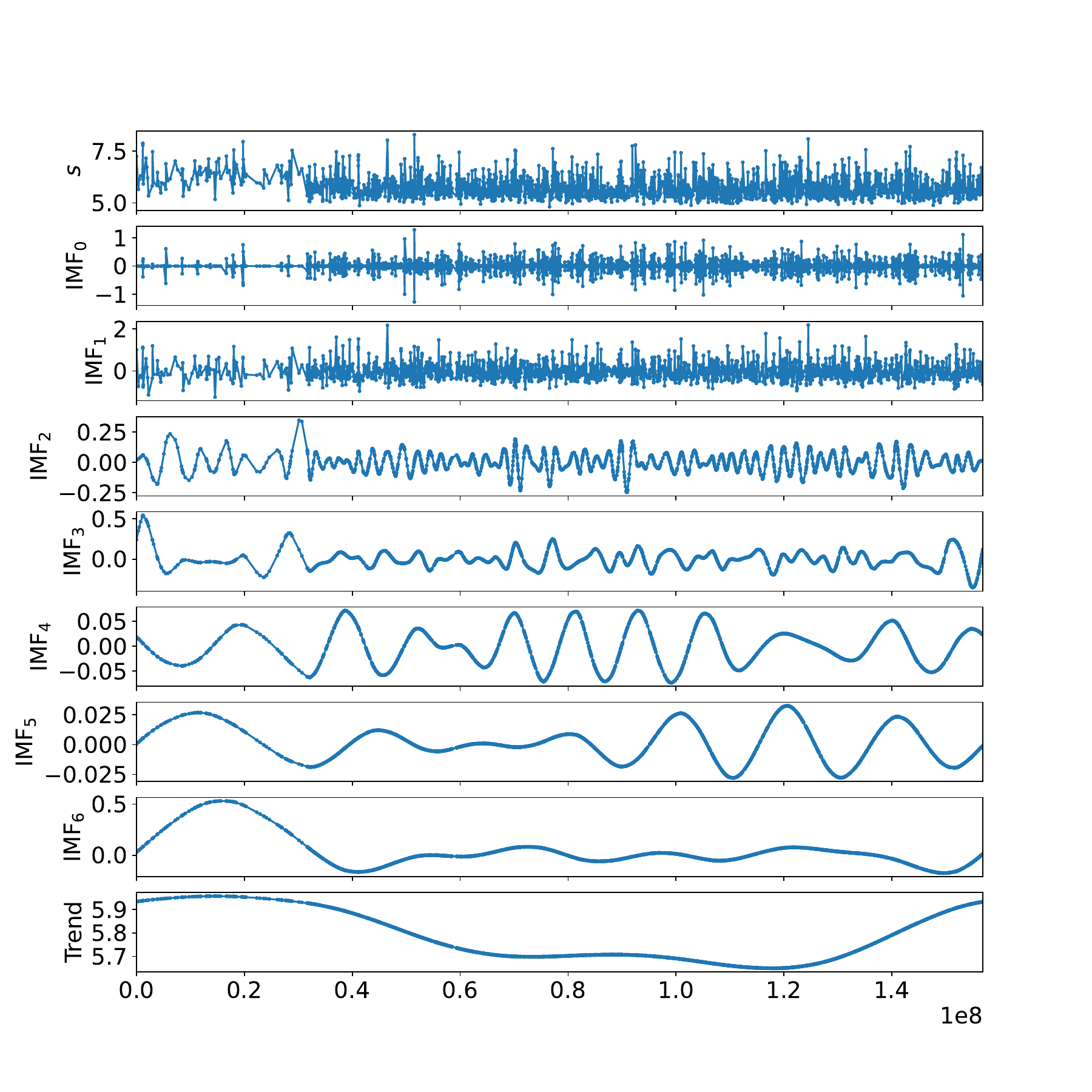}
        \caption{GFT-IF\label{fig:ex_3GFT_IMFs}}
    \end{subfigure}
    \begin{subfigure}{0.48\textwidth}
      \centering
      \includegraphics[width=\textwidth]{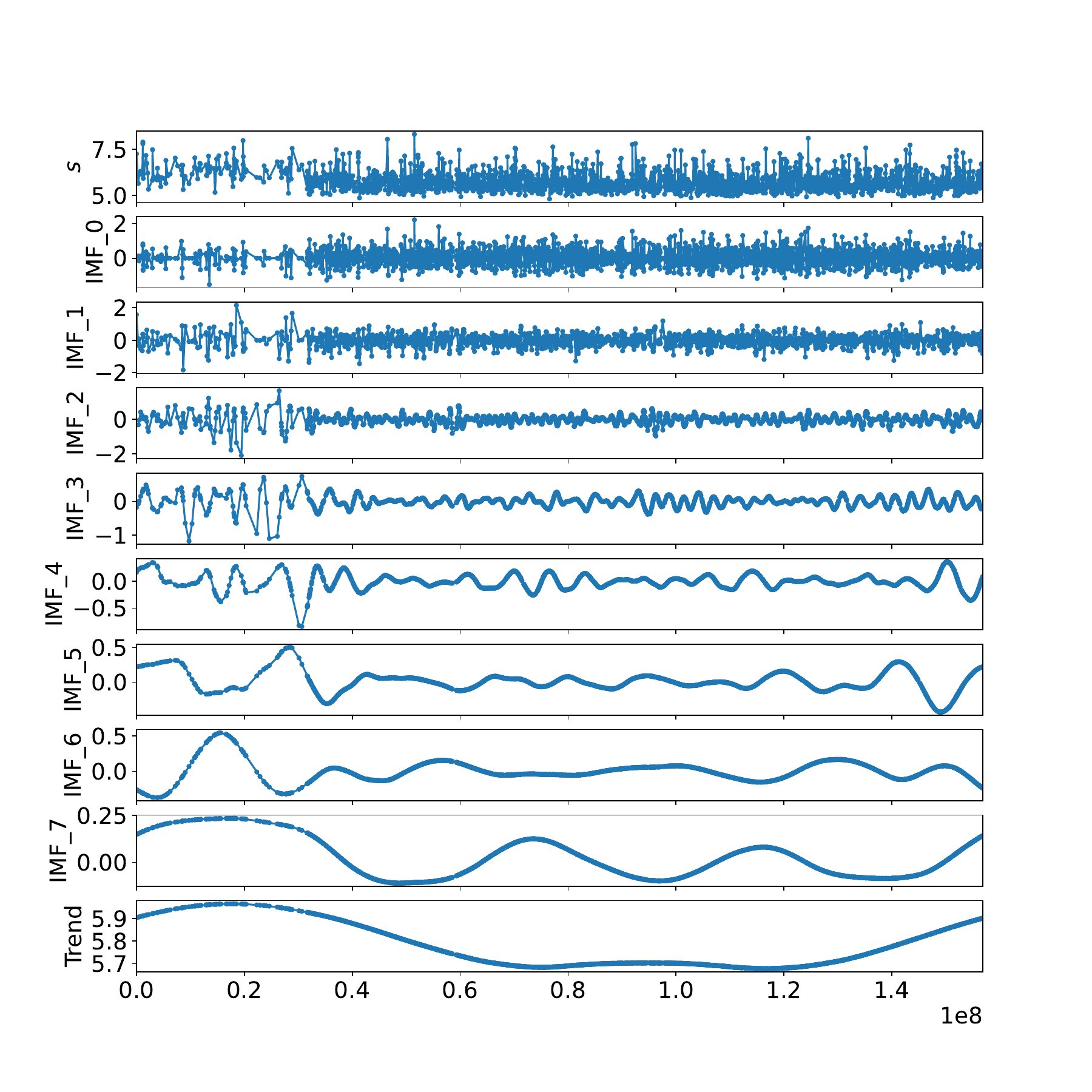}
      \caption{DB-IF\label{fig:ex_3DB_IMFs}}
    \end{subfigure}
  }
  \caption{IMFs obtained with the GFT-IF (\subref{fig:ex_3GFT_IMFs}) and DB-IF (\subref{fig:ex_3DB_IMFs}) algorithms applied to the signal representing the magnitude of earthquakes.\label{fig:ex_3IMFs}}
\end{figure}

\begin{figure}[h]
  \makebox[\textwidth][c]{
    \begin{subfigure}{0.48\textwidth}
        \centering
        \includegraphics[width=\textwidth]{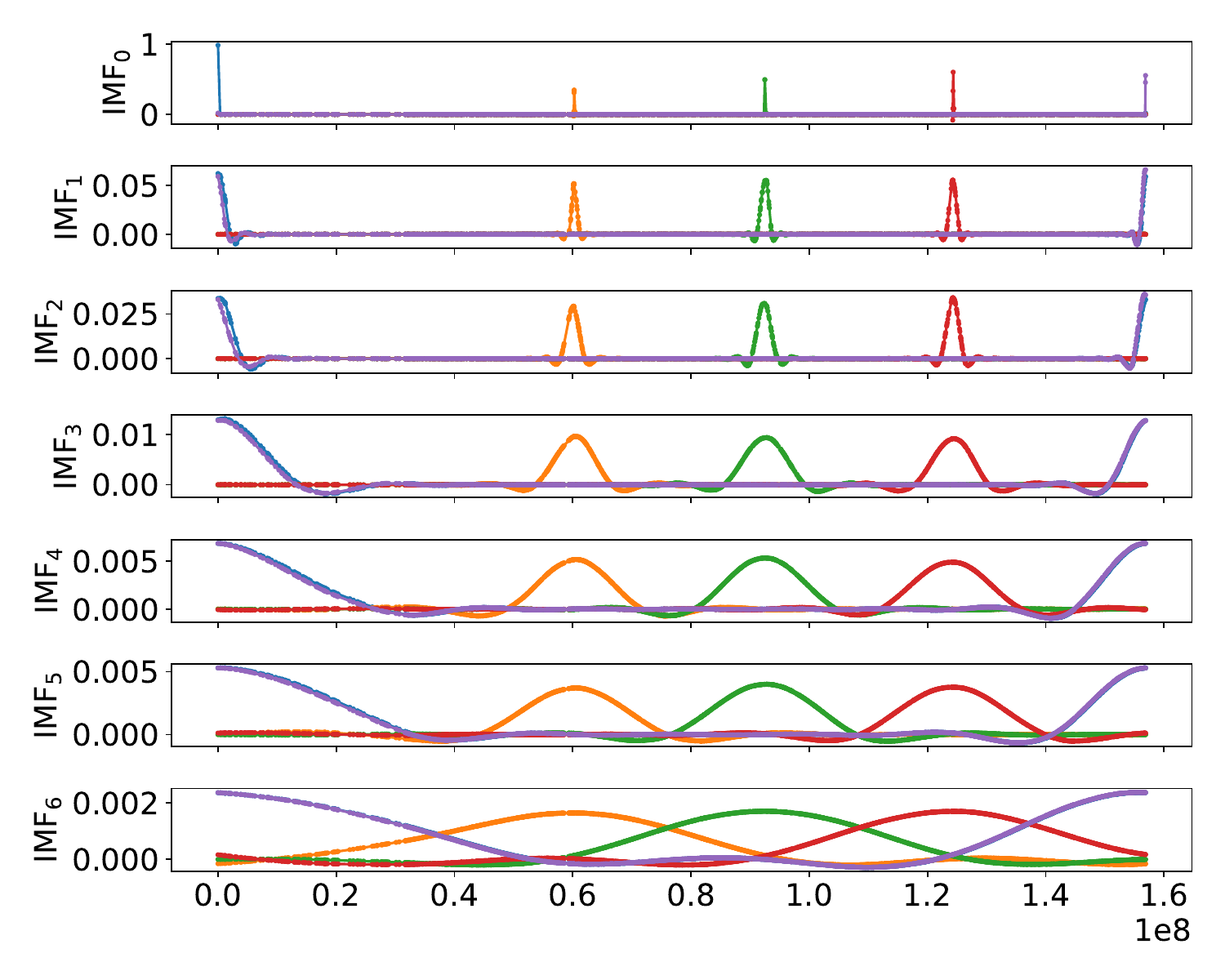}
        \caption{GFT-IF\label{fig:ex_3GFT_windows}}
    \end{subfigure}
    \begin{subfigure}{0.48\textwidth}
      \centering
      \includegraphics[width=\textwidth]{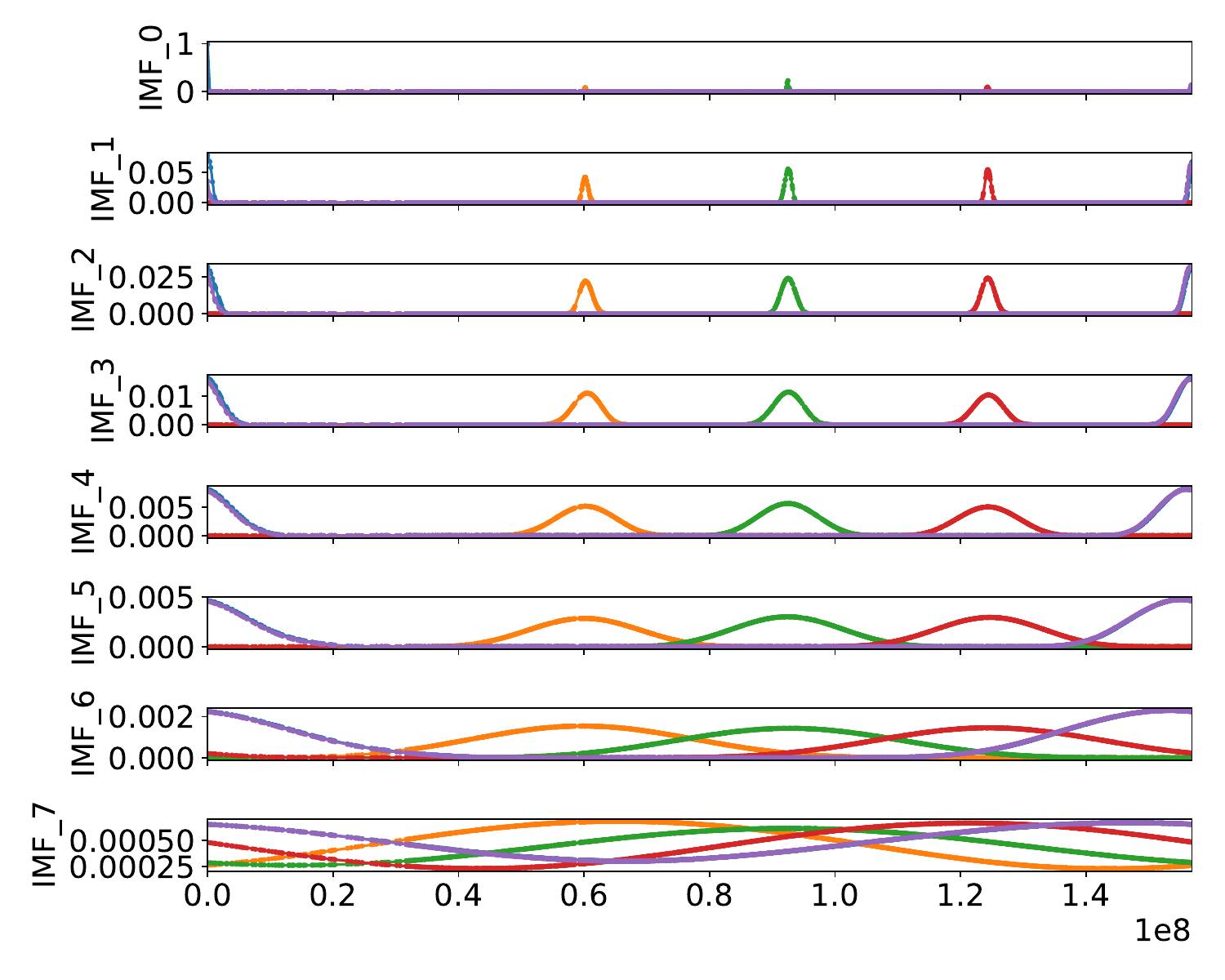}
      \caption{DB-IF\label{fig:ex_3DB_windows}}
    \end{subfigure}
  }
  \caption{Examples of windows used for each IMF extraction in the GFT-IF (\subref{fig:ex_3GFT_windows}) and DB-IF (\subref{fig:ex_3DB_windows}) algorithms for the results shown in \Cref{fig:ex_3IMFs}.\label{fig:ex_3windows}}
\end{figure}

\begin{figure}[h]
  \makebox[\textwidth][c]{
    \begin{subfigure}{0.48\textwidth}
        \centering
        \includegraphics[width=\textwidth]{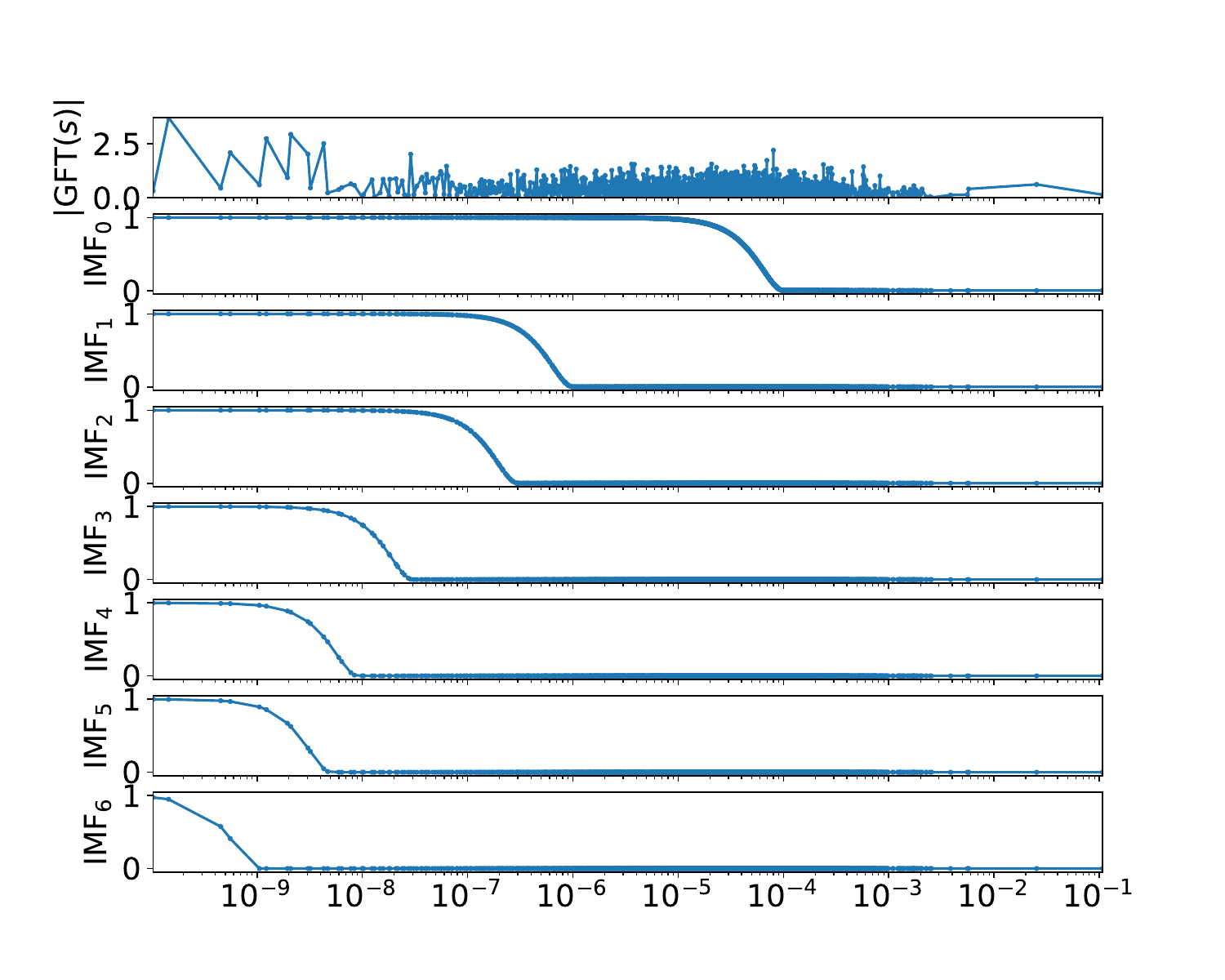}
        \caption{Window kernels\label{fig:ex_3GFT_windows_spectrum}}
    \end{subfigure}
    \begin{subfigure}{0.48\textwidth}
      \centering
      \includegraphics[width=\textwidth]{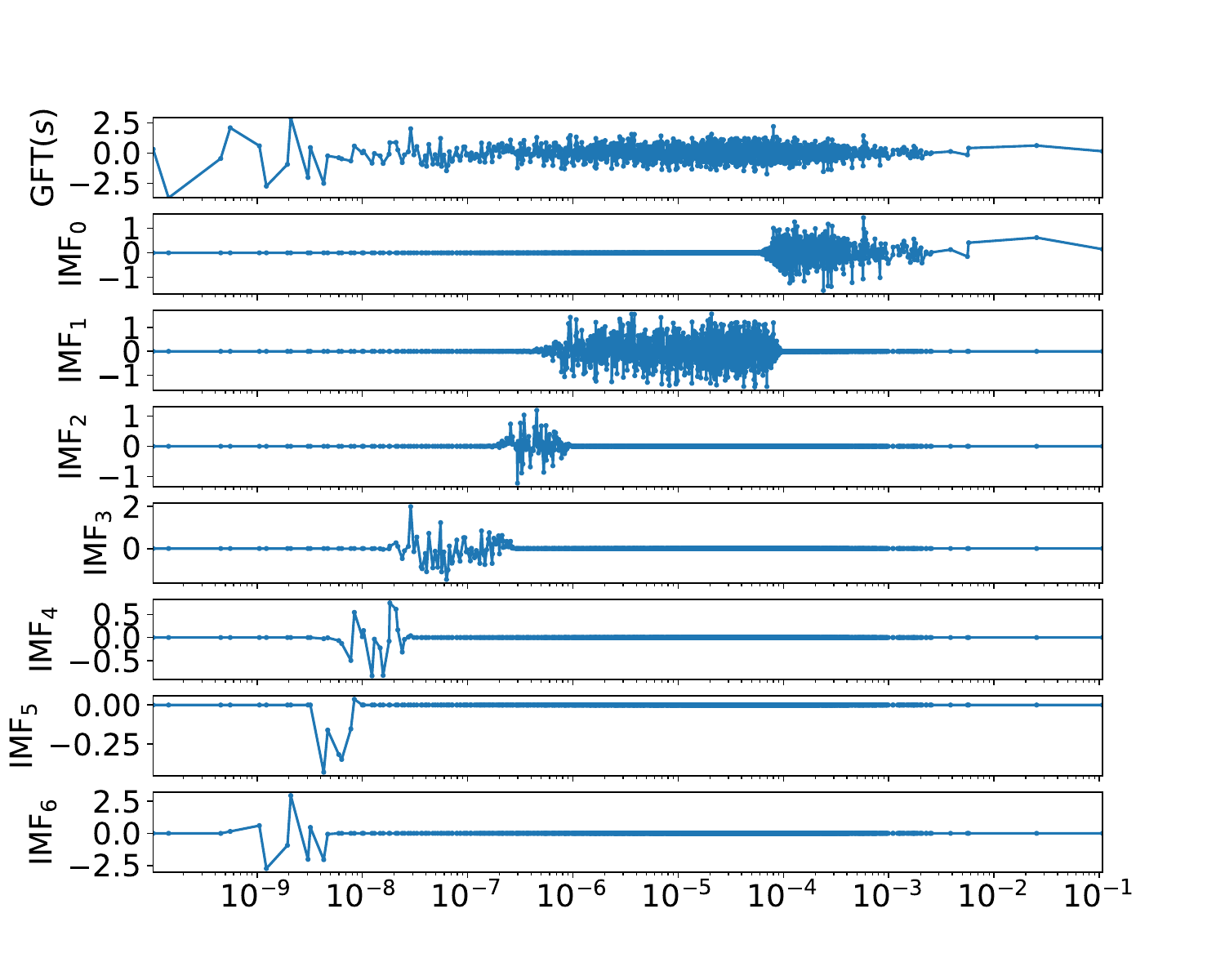}
      \caption{IMFs spectrum\label{fig:ex_3GFT_IMFs_spectrum}}
    \end{subfigure}
  }
  \caption{Window kernels used for each IMF extraction in the GFT-IF algorithm (\subref{fig:ex_3GFT_windows_spectrum}) and spectrum of the obtained IMFs (\subref{fig:ex_3GFT_IMFs_spectrum}) for the results shown in \Cref{fig:ex_3GFT_IMFs}.\label{fig:ex_3GFT_windows_and_IMFs_spectrum}}
\end{figure}

\FloatBarrier

\subsection{Example 4: Total Electron Content over Italy}
\label{sec:example_4}
In this example, we have tested the GFT-IF and the DB-IF algorithms on real satellite data representing the Total Electron Content (TEC) in the Earth's ionosphere over Italy.
Data are obtained from the analysis of signals received from GPS satellites by a network of ground stations and was collected from RING-INGV\footnote{RING-INGV: \url{https://webring.gm.ingv.it/}} and CDDIS-NASA\footnote{CDDIS-NASA: \url{https://cddis.nasa.gov/archive/gnss/data/daily/2025/}}.
The graph is constructed as the Delaunay triangulation of those points with edges weighted by $1/d(i,j)$, where $d(i,j)$ is the Euclidean distance between the vertices $i$ and $j$.
\Cref{fig:ex_4signal} shows the signal representing the TEC over Italy while, \Cref{fig:ex_4GFT_kernel} shows the GFT of
the signal $s$ compared with the convolution kernel used in the GFT-IF
algorithm.
\Cref{fig:ex_4windows} shows some examples of the window functions used in the GFT-IF and DB-IF methods.
Similarly to \Cref{sec:example_2}, the windows of both methods are qualitatively similar, but the windows of the GFT-IF algorithm have some small negative values while the windows of the DB-IF algorithm are slightly shifted towards the center of the domain.
\Cref{fig:ex_4GFT_results,fig:ex_4DB_results} shows the results of the GFT-IF and DB-IF algorithms respectively.
In both cases, the methods separate high frequency noise from a trend component.
The results obtained with GFT-IF shows a clean simple gradient in the trend component, while the results obtained with DB-IF shows a more complex structure in the trend component with some boundary effects probably due to the fact that the windows of the DB-IF algorithm are shifted towards the center of the domain.

\begin{figure}[h]
  \makebox[\textwidth][c]{
    \begin{subfigure}[t]{0.48\textwidth}
        \centering
        \includegraphics[width=\textwidth]{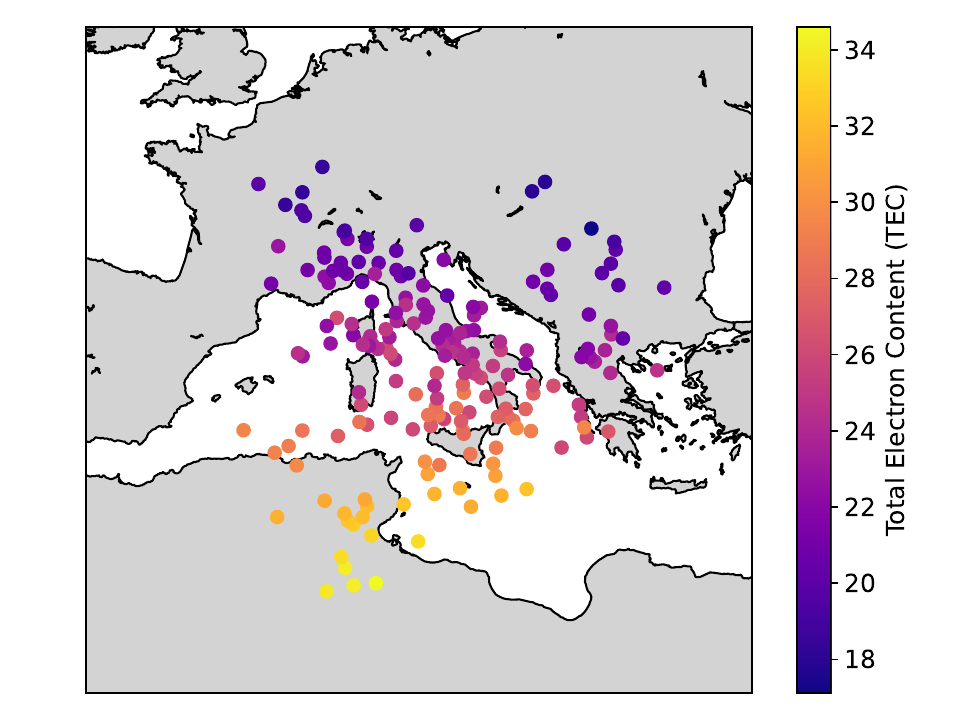}
        \caption{Signal $s$ representing the TEC over Italy.\label{fig:ex_4signal}}
    \end{subfigure}
    \begin{subfigure}[t]{0.48\textwidth}
      \centering
      \includegraphics[width=\textwidth]{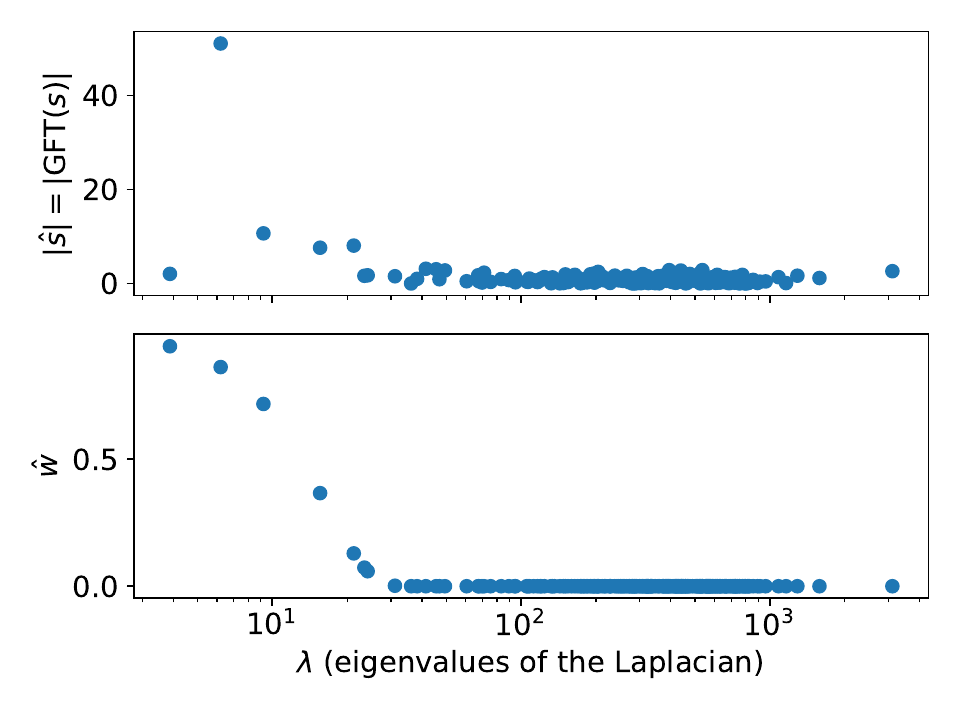}
      \caption{GFT of the signal $s$ (see \Cref{fig:ex_4signal}) compared with the convolution kernel used in the GFT-IF algorithm.\label{fig:ex_4GFT_kernel}}
    \end{subfigure}
  }
  \caption{\label{fig:ex_4signal_and_kernel}}
\end{figure}

\begin{figure}[h]
  \makebox[\textwidth][c]{
    \begin{subfigure}{0.33\textwidth}
      \centering
      \includegraphics[width=\textwidth]{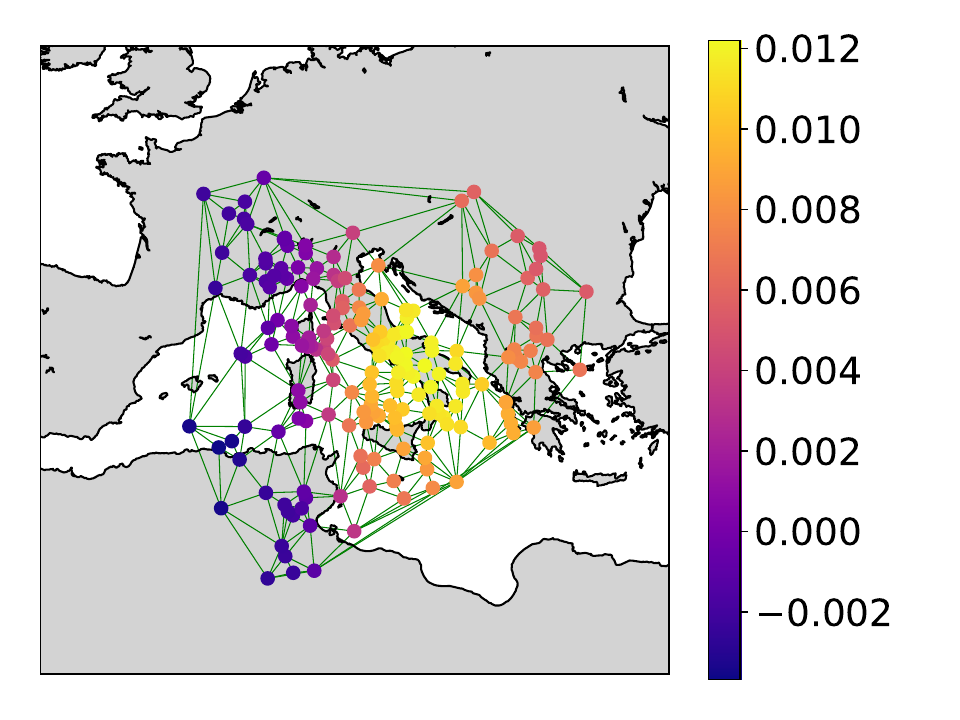}
    \end{subfigure}
    \begin{subfigure}{0.33\textwidth}
      \centering
      \includegraphics[width=\textwidth]{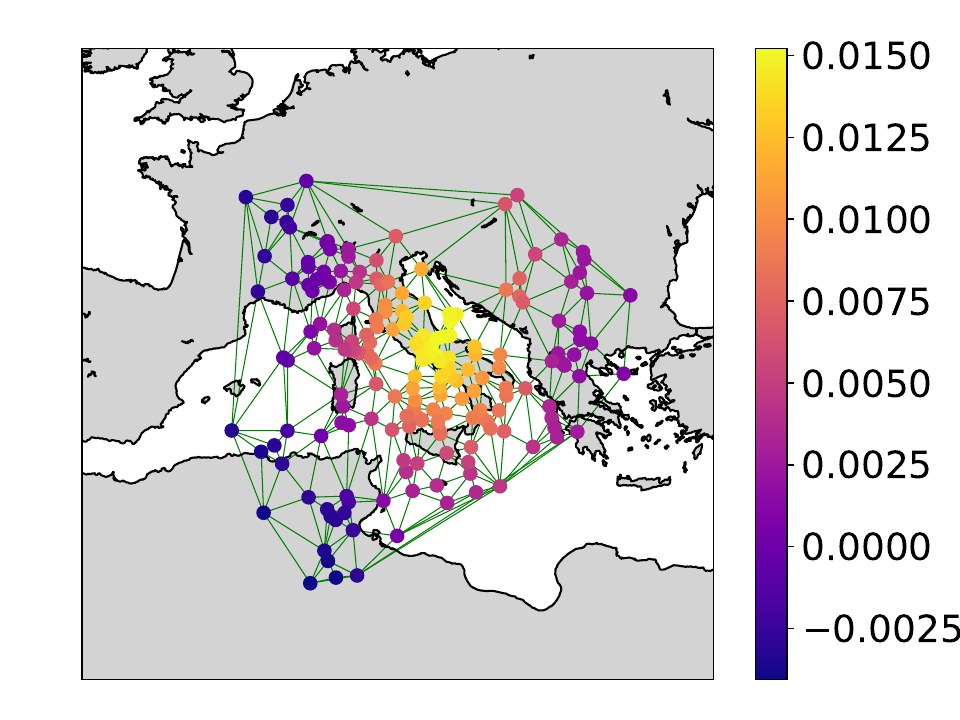}
    \end{subfigure}
    \begin{subfigure}{0.33\textwidth}
      \centering
      \includegraphics[width=\textwidth]{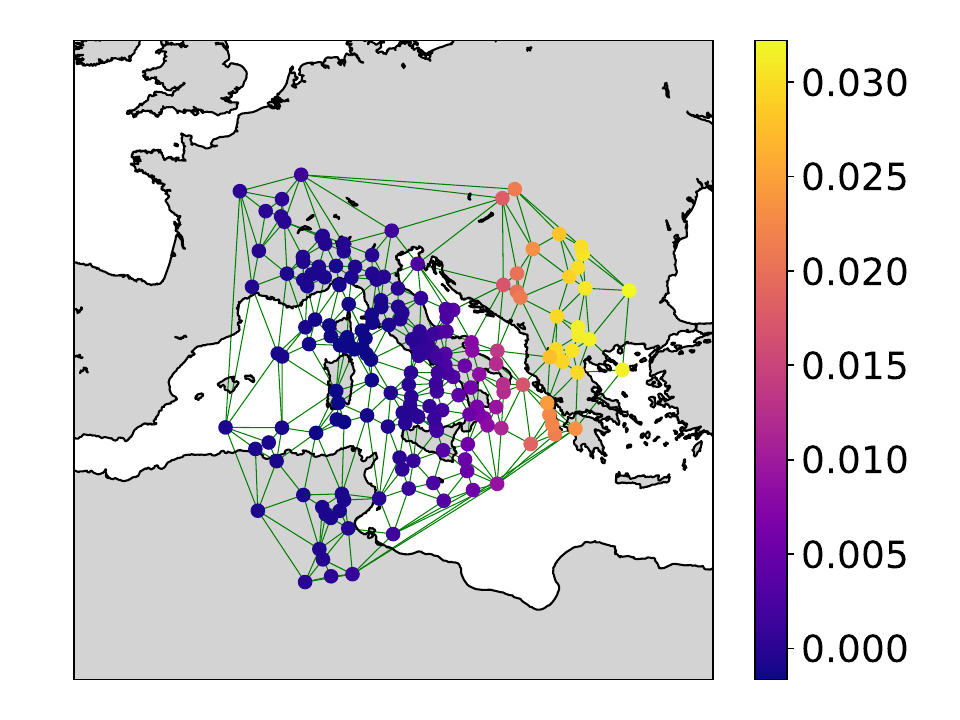}
    \end{subfigure}
  }
  \makebox[\textwidth][c]{
    \begin{subfigure}{0.33\textwidth}
      \centering
      \includegraphics[width=\textwidth]{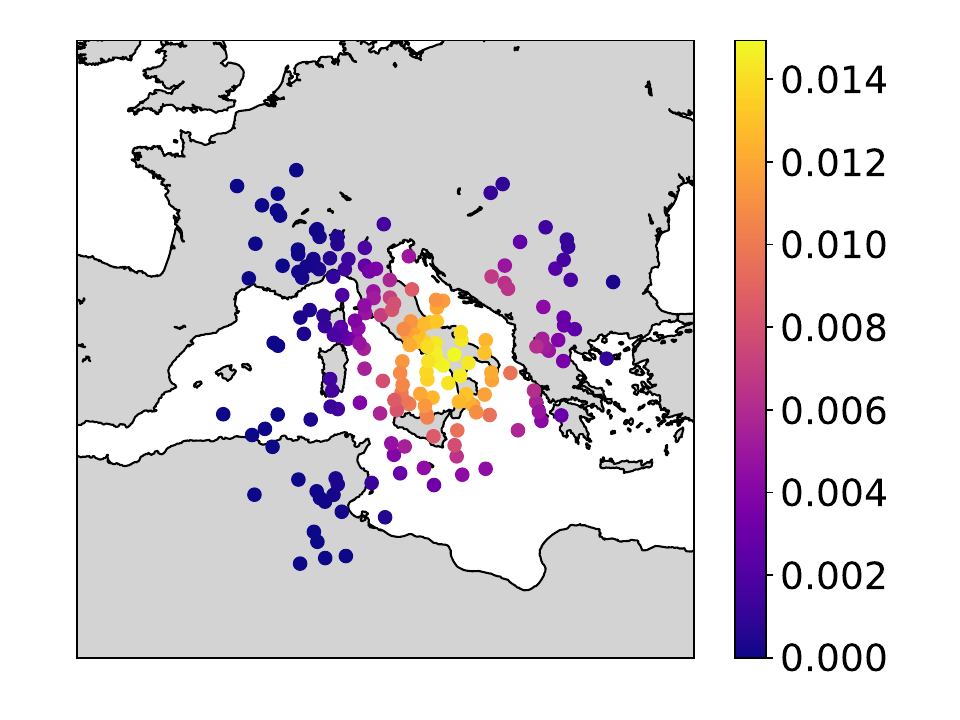}
    \end{subfigure}
    \begin{subfigure}{0.33\textwidth}
      \centering
      \includegraphics[width=\textwidth]{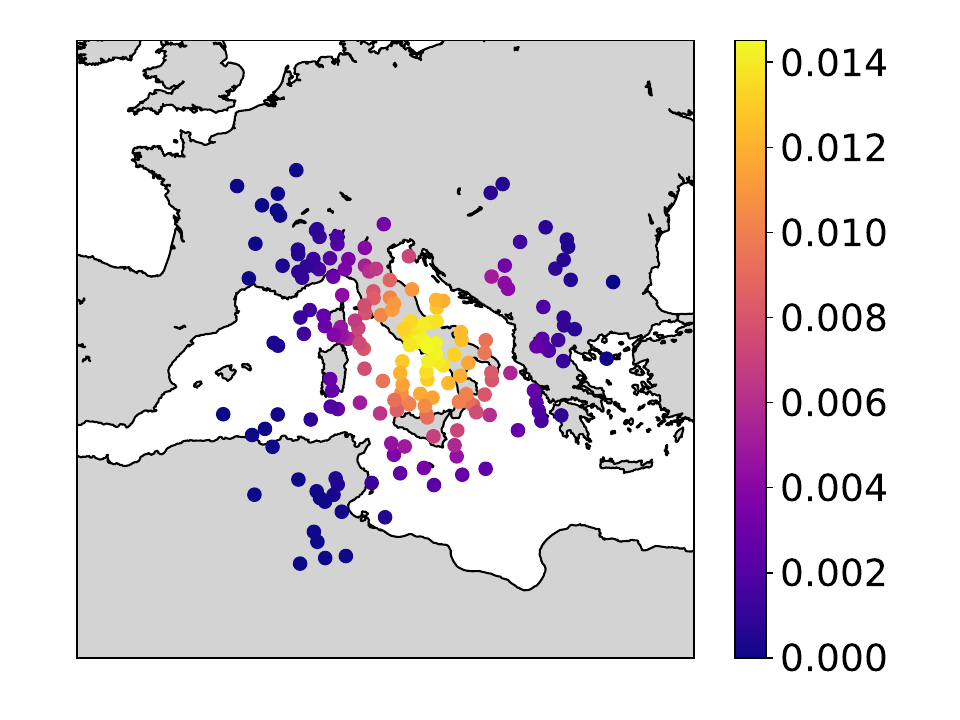}
    \end{subfigure}
    \begin{subfigure}{0.33\textwidth}
      \centering
      \includegraphics[width=\textwidth]{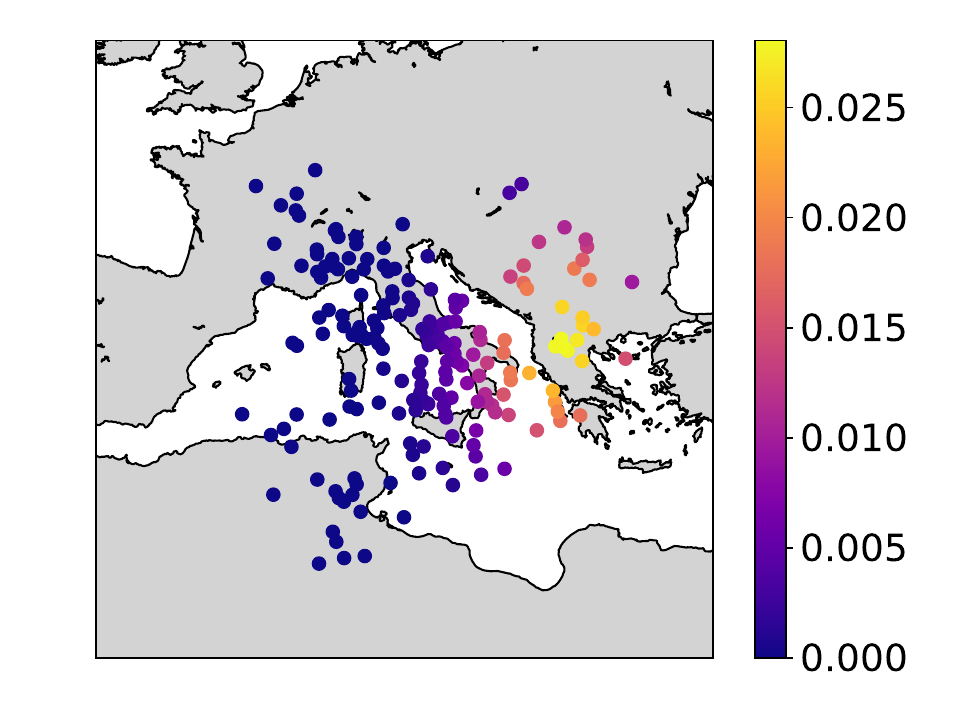}
    \end{subfigure}
  }
  \caption{Examples of the window functions used in the GFT-IF (first row) and DB-IF (second row) algorithms. In the case of the GFT-IF algorithm, those windows are obtained by convolving the kernel $w$ with delta signals centered at different vertices of the graph. In the case of the DB-IF algorithm, those windows correspond to rows of the window matrix $W$.
  \label{fig:ex_4windows}}
\end{figure}

\begin{figure}[h]
  \centering
  \makebox[\textwidth][c]{
    \begin{subfigure}{0.33\textwidth}
      \centering
      \includegraphics[width=\textwidth]{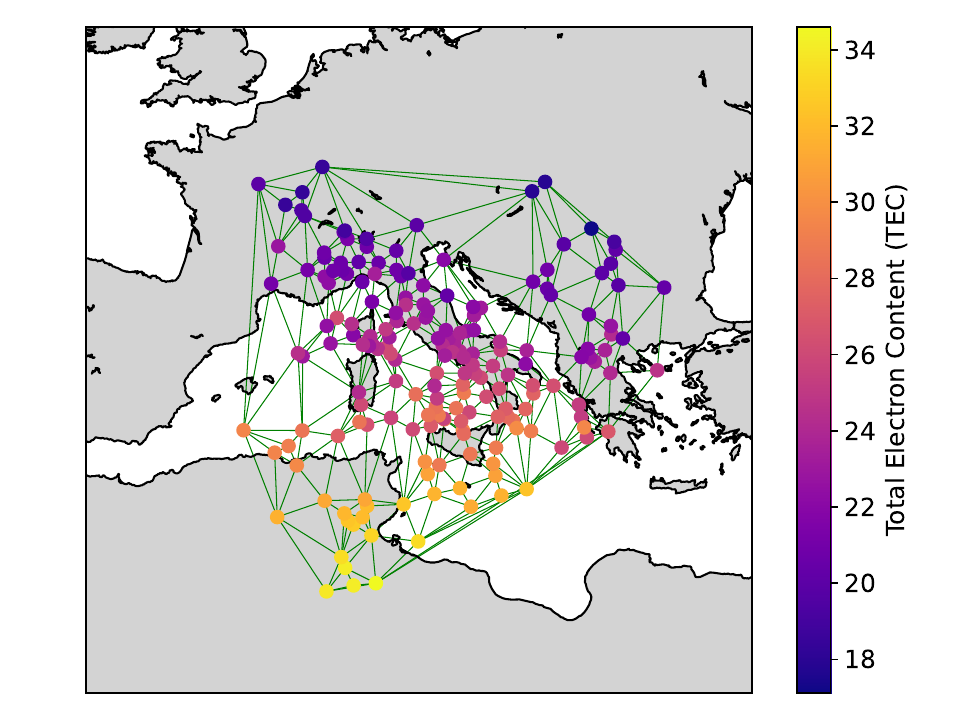}
      \caption{\label{fig:ex_4signal_GFT_results}}
    \end{subfigure}
    \begin{subfigure}{0.33\textwidth}
      \centering
      \includegraphics[width=\textwidth]{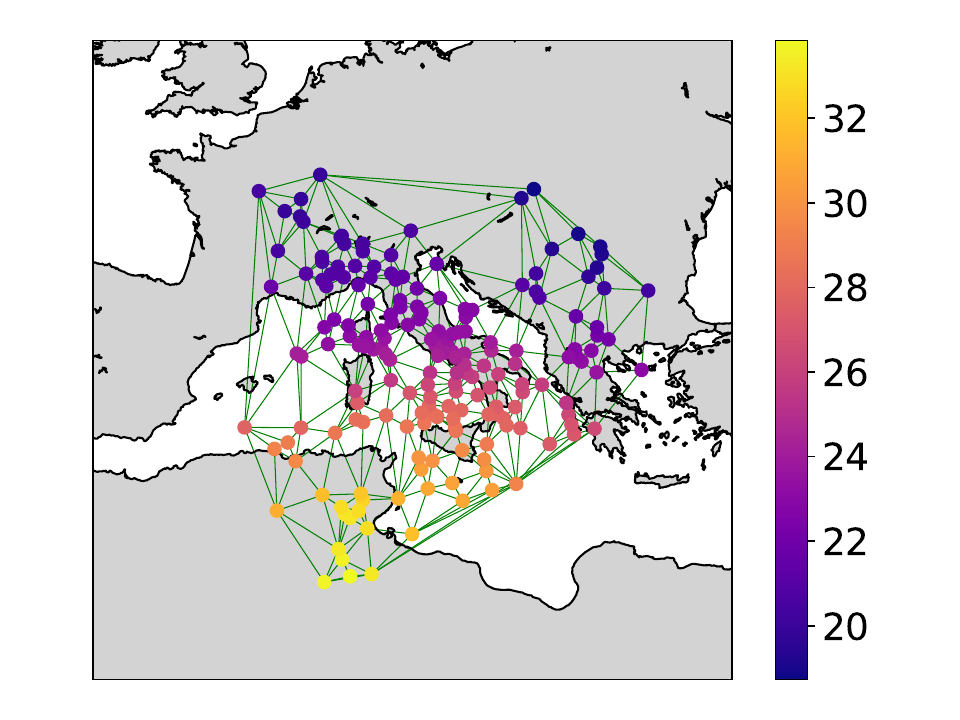}
      \caption{\label{fig:ex_4residual}}
    \end{subfigure}
    \begin{subfigure}{0.33\textwidth}
      \centering
      \includegraphics[width=\textwidth]{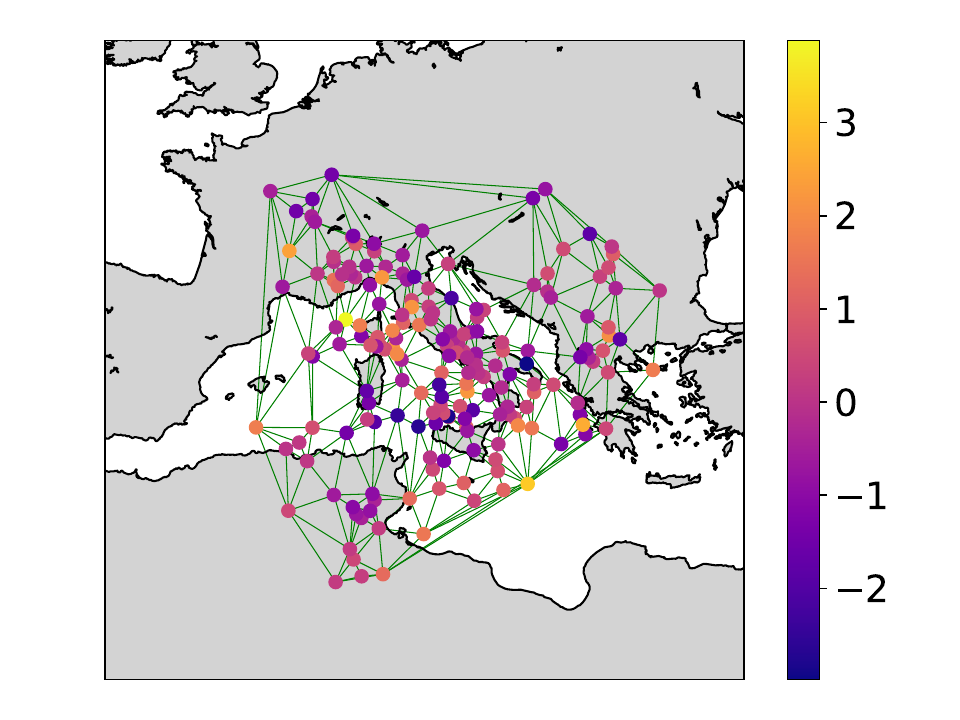}
      \caption{\label{fig:ex_4IMF_0}}
    \end{subfigure}
  }
  \caption{Results of the GFT-IF algorithm applied to the signal $s$ (\subref{fig:ex_4signal_GFT_results}). The figure~(\subref{fig:ex_4residual}) shows
  the residual signal while the figure~(\subref{fig:ex_4IMF_0}) shows the first IMF obtained with the GFT-IF algorithm.\label{fig:ex_4GFT_results}}
\end{figure}

\begin{figure}[h]
  \centering
  \makebox[\textwidth][c]{
    \begin{subfigure}{0.33\textwidth}
      \centering
      \includegraphics[width=\textwidth]{./example_4_DB_IF_signal.pdf}
      \caption{\label{fig:ex_4signal_DB_results}}
    \end{subfigure}
    \begin{subfigure}{0.33\textwidth}
      \centering
      \includegraphics[width=\textwidth]{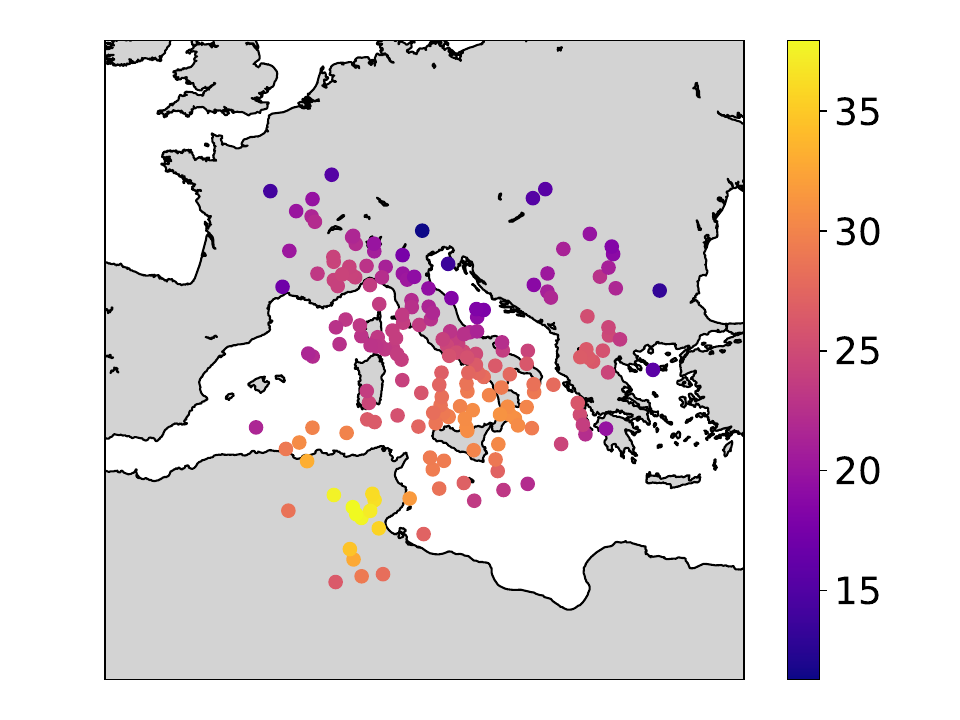}
      \caption{\label{fig:ex_4DB_residual}}
    \end{subfigure}
    \begin{subfigure}{0.33\textwidth}
      \centering
      \includegraphics[width=\textwidth]{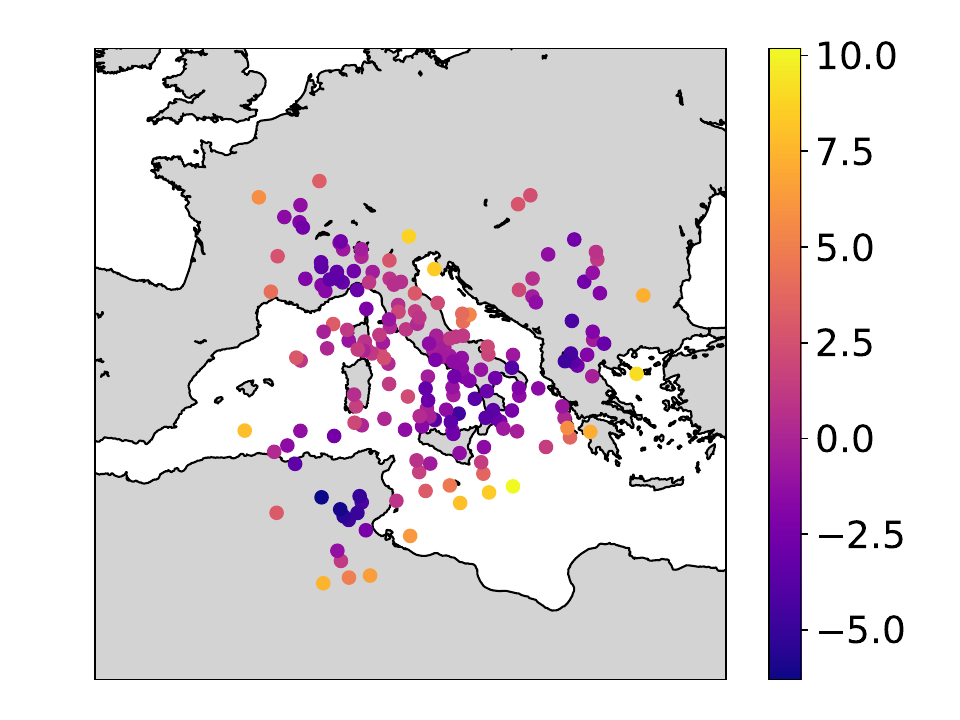}
      \caption{\label{fig:ex_4DB_IMF_0}}
    \end{subfigure}
  }
  \caption{Results of the DB-IF algorithm applied to the signal $s$ (\subref{fig:ex_4signal_DB_results}). The figure~(\subref{fig:ex_4DB_residual}) shows
  the residual signal while the figure~(\subref{fig:ex_4DB_IMF_0}) shows the first IMF obtained with the DB-IF algorithm.\label{fig:ex_4DB_results}}
\end{figure}

\FloatBarrier

\section{Conclusion}

This work extends the Iterative Filtering
algorithm to signals on graphs. We proposed two algorithms: the Graph Fourier Transform Iterative Filtering and the Distance Based Iterative Filtering.
The DB-IF algorithm uses a distance matrix to create window functions
centered at each vertex, allowing direct control over window size. The GFT-IF
algorithm, instead, uses the Graph Fourier Transform to obtain a convolution-like
operation on graphs, leveraging the graph Laplacian.

Numerical results show that both DB-IF and GFT-IF are able to
decompose signals on graphs.
In particular, in the examples that we presented, the obtained decompositions are qualitatively similar to the expected ones.

An important aspect that was not addressed in this work is the choice of the
window size for the DB-IF algorithm and the convolution kernel for the GFT-IF
algorithm. For the numerical results, we have chosen those parameters manually.

Future work should focus on analyzing whether the accuracy of the results can be improved. We plan to study the influence of the boundaries, compared with classical IF-based methods, like FIF and FIF2 \cite{MIF,sfarra2022maximizing}, and develop strategies to minimize the potential errors.
Furthermore, different strategies for choosing
the window size and the convolution kernel should be explored and compared
against each other.
In addition, graphs with different structures should be tested
to see how the algorithms perform in different scenarios.
Some interesting cases to test could be, for example, abstract graphs with no natural embedding, such as social networks, or graphs representing discretization of manifolds with more complex geometries such as spheres.
Finally, since the computation of the GFT is expensive for large graphs, approximation methods could be explored to reduce the complexity of the GFT-IF algorithm exploring approximated spectral decompositions of the graph Laplacian.

\section{Acknowledgments}

A. C., M.D., and G.S. are members of the Gruppo Nazionale Calcolo Scientifico-Istituto Nazionale di Alta Matematica (GNCS-INdAM).

A.C. was partially supported through the GNCS-INdAM Project CUP E53C23001670001, and was supported by the Italian Ministry of the University and Research and the European Union through the ``Next Generation EU'', Mission 4, Component 1, under the PRIN PNRR 2022 grant number CUP E53D23018040001 ERC field PE1 project P2022XME5P titled ``Circular Economy from the Mathematics for Signal Processing prospective''.

M.D., and G.S. were partially supported by the Italian Ministry of the University through the PRIN 2022 project ``Inverse Problems in the Imaging Sciences (IPIS)'' (2022ANC8HL).

\FloatBarrier

\bibliographystyle{cas-model2-names}

\bibliography{bibliography}

\end{document}